\documentclass[12pt]{article}
\usepackage[T1]{fontenc}
\usepackage{amssymb}
\usepackage{amsfonts}
\usepackage{amsmath}
\usepackage{amsbsy}
\usepackage{graphicx}
\usepackage{theorem}
\usepackage{hyperref}
\newtheorem{theorem}{Theorem}[section]
\newtheorem{lemma}[theorem]{Lemma}
\newtheorem{corollary}[theorem]{Corollary}
\newtheorem{definition}[theorem]{Definition}
\newtheorem{proposition}[theorem]{Proposition}
\newtheorem{remark}[theorem]{Remark}

\numberwithin{equation}{section}
\def\beq{\begin{equation}\displaystyle}
\def\eeq{\end{equation}}
\def\bel{\begin{equation} \displaystyle \begin{array}{l} }
\def\eel{\end{array} \end{equation} }
\def\bell{\begin{equation} \displaystyle \begin{array}{ll}  }
\def\eell{\end{array} \end{equation} }

\def\bea{\begin{eqnarray}}
\def\eea{\end{eqnarray} }
\def\bean{\begin{eqnarray*}}
\def\eean{\end{eqnarray*} }
\newenvironment{proof}{\noindent{\bf Proof.~}}
{{\mbox{}\hfill {\small \fbox{}}\\}}
\def\qed{\mbox{}\hfill {\small \fbox{}}\\}
\catcode`@=11
\renewcommand\appendix{\bigskip {\noindent \Large \bf Appendix}
  \setcounter{section}{0}%
  \setcounter{subsection}{0}%
\setcounter{equation}{0}%
\setcounter{theorem}{0}%
\def\thetheorem{A.\arabic{theorem}}
\def\theequation {A.\arabic{equation}}}
\catcode`@=12
\def\NN{\mathbb{N}}

\def\RR{\mathbb{R}}

\def\ds{\displaystyle}

\def\bs{\bigskip}

\def\eps{\varepsilon}
\def\bar#1{{\overline #1}}
\def\pa{\partial}

\def\calM{{\cal M}}

\def\smes{{\cal S}_{\cal M}}
\def\Dpetit{{\mbox{\tiny $\Delta$}}}
\def\achapo{\widehat{a}}
\def\bchapo{\widehat{b}}

\begin{document}

\bs
\begin{center}
\Large \bf Chemotaxis: from kinetic equations to aggregate dynamics 
\end{center}

\bs
\begin{center}
F. James$^a$ and N. Vauchelet$^{b}$

\bs

 \medskip

 {\footnotesize $^a$ 
Math\'ematiques -- Analyse, Probabilit\'es, Mod\'elisation -- Orl\'eans (MAPMO), \\
Universit\'e d'Orl\'eans \& CNRS UMR 6628, \\
F\'ed\'eration Denis Poisson, Universit\'e d'Orl\'eans \& CNRS FR 2964, \\
45067 Orl\'eans Cedex 2, France}
 \\

 {\footnotesize $^b$ 
UPMC Univ Paris 06, UMR 7598, Laboratoire Jacques-Louis Lions, \\
CNRS, UMR 7598, Laboratoire Jacques-Louis Lions and \\
INRIA Paris-Rocquencourt, Equipe BANG \\
F-75005, Paris, France}
 \\

 \medskip
 {\footnotesize {\em E-mail addresses:} 
 {\tt francois.james@univ-orleans.fr, vauchelet@ann.jussieu.fr}
 }

\end{center}

\bs

\begin{abstract}
The hydrodynamic limit for a kinetic model of chemotaxis is investigated. The limit equation is a non local conservation law,
for which finite time blow-up occurs, giving rise to measure-valued solutions and discontinuous velocities. An adaptation of
the notion of duality solutions, introduced for linear equations with discontinuous coefficients, leads to an existence result.
Uniqueness is obtained through a precise definition of the nonlinear flux as well as the complete dynamics of aggregates, i.e.
combinations of Dirac masses. Finally a particle method is used to build an adapted numerical scheme.
\end{abstract}

\bs



\bs

%%%%%%%%%%%%%%%%%%%%%%%%%%%%%%%%%%%%
\section{Introduction}
%%%%%%%%%%%%%%%%%%%%%%%%%%%%%%%%%%%%

Kinetic frameworks have been investigated to describe the chemotactic movement of cells in 
the presence of a chemical substance 
since in the 80's experimental observations showed that the 
motion of bacteria (e.g. {\it Escherichia Coli}) is due to the alternation of `runs and tumbles'.
The so-called Othmer-Dunbar-Alt model \cite{alt,erbanothmerbacterie,othdunalt,othste} describes
the evolution of the distribution function of cells at time $t$, position $x$ and velocity
$v$, assumed to have a constant modulus $c>0$, as well as the concentration $S(t,x)$ of the involved chemical.
A general formulation for this model can be written as
\beq\label{cinetique}
\left\{
\begin{array}{rcl}
\ds \pa_t f_\eps + v\cdot \nabla_x f_\eps &=& \ds\frac{1}{\eps} \int_{|v'|=c} \big(T[S_\eps](v'\to v) f_\eps(v') - T[S_\eps](v\to v') f_\eps(v)\big) \,dv',    \\
\ds -\Delta S_\eps + S_\eps &=& \rho_\eps(t,x) := \ds\int_{|v|=c} f_\eps(t,x,v)\,dv.
\end{array}\right.
\eeq
The second equation describes the dynamics of the chemical agent which diffuses in the domain. It is produced by the cells themselves 
with a rate proportional to the density of cells $\rho$ and disappears with a rate proportional to $S$. 
The transport operator on the left-hand side of the first equation stands for the unbiased movement of cells (`runs'), while the
right-hand side governs `tumbles', that is chemotactic orientation, or taxis, through the 
turning kernel $T[S](v' \to v)$, which is the rate of cells changing their velocity from $v'$ to $v$. 

The parameter $\eps$ corresponds to the time interval of information sampling for the bacteria, usually $\eps\ll1$, and when it goes to zero, one
expects to recover the collective behaviour of the population, that is a macroscopic equation for the density $\rho(t,x)$ of cells.
Such derivations have been proposed by several authors. 
When the taxis is small compared to the unbiased movement of cells, the scaling must be of diffusive type, so that the limit
equations are of diffusion or drift-diffusion type, see for instance \cite{chalubperth}
for a rigorous proof. In \cite{hillenothmer, othhill}, 
the authors show that the classical Patlak-Keller-Segel model can 
be obtained in a diffusive limit for a given smooth chemoattractant
concentration.

We focus here on the opposite case, that is when taxis dominates the unbiased movements. This is accounted for in the model 
by the choice of the scaling in equation \eqref{cinetique}.
Moreover, we consider positive chemotaxis, which means that the involved chemical is attracting cells, and therefore is called
chemoattractant.
The model has been proposed in \cite{dolschmeis}, several works have been devoted to the mathematical study of this kinetic
system. Existence of solutions has been obtained for various assumptions
on the turning kernel in \cite{chalubperth,bourncalv,erbanhwang,hwang}.
Numerical simulations of this system are proposed in \cite{nv}.
The limit problem is usually of hyperbolic type, see for instance \cite{filblaurpert, perthame, perthame2} for a hyperbolic limit
 model  which consists in a conservation equation for the cell density and a momentum balance equation. 

It is not difficult to obtain the following formal hydrodynamic limit to equation \eqref{cinetique}, more precisely on the total density of particles
$\rho=\lim_\eps\rho_\eps$: 
\begin{equation}
\ds \pa_t \rho + {\mathop{\rm div}}_x \big(a[S] \rho\big) = 0, \qquad
\ds -\Delta_{xx} S +S = \rho.
\label{Hydro1}\end{equation}
Here the macroscopic velocity $a[S]$ depends on the chemoattractant concentration $S$ through the turning kernel.
This system of equations has been obtained in \cite{dolschmeis}, with a rigorous  proof in the two-dimensional setting for a fixed smooth $S$,
and therefore a bounded density $\rho$. 
The aim of this paper is to obtain rigorously this limit for the whole coupled system. 
Severe difficulties arise then mainly due to the lack of estimates
for the solutions to the kinetic model when $\eps$ goes to zero and
consequently to the very weak regularity of the solutions 
to the limit problem. 

It turns out that the limit equation is in some sense 
a weakly nonlinear conservation equation on the density $\rho$.
Indeed the expected velocity field depends on $\rho$, 
but through $S$, and therefore in a non local way. Actually it can be written
as a variant of the so-called aggregation equation, for which blow-up in finite time
is evidenced (see e.g. \cite{bertozzi2}), leading to measure-valued solutions.
In this respect, this equation behaves also like
linear equations with discontinuous coefficients. 
In particular Dirac masses can arise,
this is the mathematical formulation of the aggregation of bacteria.
Therefore $S$ is no longer smooth, and a major difficulty
in this study will be to define properly the velocity field 
$a=a[S]$ and the product $a\rho$.

The viewpoint of the aggregation equation has been extensively studied by Carrillo {\sl et al.} \cite{Carrillo}
through optimal transport techniques. Existence and uniqueness are obtained in a very weak sense, and the
dynamics of aggregates is also given. We propose here another approach, 
based on the notion of duality solutions, as introduced in the linear case by Bouchut and James \cite{bj1}.
The main drawback is that presently we have to restrict
ourselves to the one-dimensional case, since the theory in higher 
dimensions is not complete yet (see \cite{bjm}).
The approach proposed by Poupaud and Rascle \cite{pouras}, 
which coincides with duality in the 1-d case, could also be explored.
Notice however, that the properties of the expected velocity field $a$ in the two-dimensional
case are not obvious either.

More precisely, we propose to proceed in a similar way as in \cite{BJpg}, 
where the nonlinear system of zero pressure gas dynamics is interpreted as a 
system of two linear conservation equations coupled through 
the definition of the product. This last point turns out to be crucial
in order to obtain a proper uniqueness result for the system \eqref{Hydro1}.
In this work, the product $a\rho$ will be defined thanks to the 
limiting flux of the kinetic system \eqref{cinetique}
(see also \cite{note} for another application of the same idea).
As we shall see, this is closely related to the dynamics of aggregates, that is combinations of Dirac masses,
which reflect some kind of collective behaviour of the population. Finally, an important application of this aggregate dynamics
is the development of a numerical scheme, based on a particle method. The motion and collapsing
of Dirac masses is clearly evidenced.

The paper is organized as follows. In Section \ref{models} we precisely state the model.
Section \ref{SecDuality} is devoted to the notion of duality solutions, and contains the main results of this article.
Some technical properties
which will be useful for the rest of the paper are given in Section \ref{SecPropS}.
Then we investigate in Section \ref{SecExistDual} the proof of the 
existence and uniqueness result of duality solution for system
\eqref{eqrhohydro}--\eqref{eqShydro} stated in Theorem \ref{ExistFlux}. 
In Section \ref{SecConvKinet} we prove the rigorous derivation of the hydrodynamical system from 
the kinetic system.
Finally, the dynamics of aggregates and the numerical scheme for the limit equation are described 
in the last section, where numerical illustrations are also provided.

%%%%%%%%%%%%%%%%%%%%%%%%%%%%%%%%%%%%%%%%%%%
\section{Modelling}\label{models}
%%%%%%%%%%%%%%%%%%%%%%%%%%%%%%%%%%%%%%%%%%%
From now on we focus on the one dimensional version of the problem, so that $x\in\RR$. 
We first recall the main assumptions leading to the kinetic equation, next we proceed to the formal 
limit.

\subsection{Kinetic model}
%%%%%%%%%%%%%%%%%%%%%%%%%%
In this work, cells are supposed to be large enough to sense
the gradient of the chemoattractant instantly. Therefore
the turning kernel takes the form (independent on $v$)
	\beq\label{T}
T[S](v'\to v) = \Phi(v'\pa_x S).
	\eeq
The function $\Phi$ is the turning rate, obviously it has to be positive.
More precisely, for attractive chemotaxis, the turning rate is smaller 
if cells swim in a favourable direction, that is 
$v\cdot \nabla_x S\ge 0$. 
Thus $\Phi$ should be a non increasing function.
A simplified model for this phenomenon is the following choice for
$\Phi$: we fix a positive parameter $\alpha$, a mean turning rate $\phi_0>0$
and take
\beq\label{phisym}
\Phi(x)=\phi_0\big(1+\phi(x)\big),
\eeq
where $\phi$ is an odd function such that
\beq\label{phi}
\phi\in C^\infty(\RR), \quad \phi' \leq 0, \quad
\phi(x)= \begin{cases}	+\lambda & \mbox{ if } x<-\alpha, \\
			-\lambda & \mbox{ if } x>\alpha,
	\end{cases}
\eeq
where $0<\lambda<1$ is a given constant. 

Now since the transport occurs in $\RR$ the set of velocities is $V = \{-c,c\}$, and the expression 
of the turning kernel simplifies in such a way that \eqref{cinetique} rewrites
	\beq\label{cin1D}
\pa_t f_\eps + v \pa_x f_\eps = \frac{1}{\eps} (\Phi(- v\pa_x S) f_\eps(-v) - 
\Phi(v\pa_x S) f_\eps(v)), \qquad v\in V.
	\eeq
	\beq\label{ellip1D}
-\pa_{xx}S_\eps + S_\eps = \rho_\eps = f_\eps(c) + f_\eps(-c).
	\eeq

The existence of weak solutions in a $L^p$ setting for a slightly different system in a more general 
framework has been obtained for instance in \cite{bourncalv,hwang}. 
Concerning precisely this model, we refer to \cite{nv} for the existence theory in any space dimension.
Notice that no uniform $L^\infty$ bounds can be expected. The reader is referred to \cite{nv} for some numerical evidences of this phenomenon, 
which is the mathematical translation of the concentration of bacteria. This is some kind of ``blow-up in infinite time'',
which for $\eps=0$ leads to actual blow-up in finite time, and creation of Dirac masses.
Moreover the balanced distribution vanishing the right 
hand side of \eqref{cin1D} depends on $S_\eps$; thus the techniques 
developed e.g. in \cite{chalubperth} cannot be applied.

\subsection{Formal hydrodynamic limit}\label{ss-sec:limit}
%%%%%%%%%%%%%%%%%%%%%%%%%%%%%%%%%%%%%%%%%%%%%%%%%%%
We formally let $\eps$ go to $0$ assuming that $S_\eps$ and $f_\eps$ admit a Hilbert expansion
	$$
f_\eps = f_0 + \eps f_1 + \cdots, \qquad S_\eps = S_0 + \eps S_1 + \cdots
	$$
Multiplying \eqref{cin1D} by $\eps$ and taking $\eps=0$, we find 
\beq\label{equilib}
\Phi(- c\pa_x S_0) f_0(-c) = \Phi(c\pa_x S_0) f_0(c).
\eeq
Summing equations $\eqref{cin1D}$ for $c$ and $-c$, we obtain
\beq\label{moment}
\pa_t (f_\eps(c) + f_\eps(-c)) + c \pa_x (f_\eps(c) - f_\eps(-c)) = 0.
\eeq
Moreover, from equation \eqref{equilib} we deduce that
$$
f_0(c)-f_0(-c) = \frac{\Phi(- c\pa_x S_0)-\Phi(c\pa_x S_0)}
{\Phi(- c\pa_x S_0)+\Phi(c\pa_x S_0)} (f_0(c) + f_0(-c)).
$$
The density at equilibrium is defined by $\rho := f_0(c) + f_0(-c)$. 
Taking $\eps = 0$ in \eqref{moment} we finally obtain
$$
\pa_t \rho + \pa_x (a(\pa_x S_0) \rho) = 0,
$$
where $a$ is defined by
$$
a(\pa_x S_0) = c\, \frac{\Phi(- c\pa_x S_0)-\Phi(c\pa_x S_0)}
{\Phi(- c\pa_x S_0)+\Phi(c\pa_x S_0)}= -c\,\phi(c\pa_x S_0),
$$
and we have used \eqref{phisym} for the last identity.
Notice that $a$ is actually a macroscopic quantity, since it is the simplified formulation of
	$$
a(\pa_xS_0) = -\frac {\int_V v\,\Phi(v\pa_xS_0)\,dv}
{\int_V \Phi(v\pa_xS_0)\,dv}
	$$
in the one-dimensional context.

We couple this equation with the limit of the elliptic problem 
\eqref{ellip1D} for the chemoattractant concentration, so that, 
in summary, and dropping the index $0$, the formal hydrodynamic 
limit is the following system
\begin{eqnarray}
&&\ds \pa_t \rho + \pa_x (a(\pa_x S) \rho) = 0, 
\label{eqrhohydro} \\[2mm]
&&\ds a(\pa_xS)= -c\,\phi(c\pa_xS),
\label{eqahydro} \\[2mm]
&&\ds -\pa_{xx} S +S = \rho, 
\label{eqShydro}
\end{eqnarray}
complemented with the boundary conditions
\beq\label{bordhydro}
\rho(t=0,x)=\rho^{ini}(x), \qquad
\lim_{x\to \pm \infty} \rho(t,x) = 0, \qquad 
\lim_{x\to \pm \infty} S(t,x)=0. 
\eeq

We now give the precise formulation of the limit system in terms of aggregate equation.
Noticing that a solution to \eqref{eqShydro} has
 the explicit expression
	\beq\label{Sexplicit}
S(t,x) = K*\rho(t,.)(x),\quad\mbox{where }K(x) = \frac{1}{2}e^{-|x|},
	\eeq
the macroscopic conservation equation for $\rho$ \eqref{eqrhohydro} can be rewritten
	$$
\pa_t \rho + \pa_x(a(\pa_x K * \rho) \rho) = 0.
	$$
When $a$ is the identity function, this is exactly the so-called aggregation equation, and since the potential
is  non-smooth, blow-up in finite time is expected. We refer the reader to e.g. \cite{bertozzi2,Carrillo}, 
and \cite{jamesnv} in the context of chemotaxis.

Similar problems were encountered for instance in \cite{NPS}, where the authors investigate the high field limit
of the Vlasov-Poisson-Fokker-Planck model in one space dimension.
The limit system is a scalar conservation law coupled to the Poisson equation,
and a proper definition of the product is needed to pass to the limit. This definition 
has been extended in two dimensions by Poupaud \cite{poupaud} using 
defect measures but losing uniqueness.

%%%%%%%%%%%%%%%%%%%%%%%%%%%%%%%%%%%%%%%%%%%%%%%%%%%%%%%
\section{Duality solutions}\label{SecDuality}
%%%%%%%%%%%%%%%%%%%%%%%%%%%%%%%%%%%%%%%%%%%%%%%%%%%%%%%
\subsection{Notations}

Let $C_0(Y,Z)$ be the set of continuous functions from $Y$ to $Z$ 
that vanish at infinity and $C_c(Y,Z)$ the set of continuous 
functions with compact support from $Y$ to $Z$.
All along the paper, we denote $\calM_{loc}(\RR)$ the space 
of local Borel measures on $\RR$. For $\rho\in {\cal M}_{loc}$
we denote by $|\rho|(\RR)$ its total variation.
We will denote $\calM_b(\RR)$ the space of measures in $\calM_{loc}(\RR)$
whose total variation is finite.
From now on, the space of measure-valued function $\calM_b(\RR)$ is
always endowed with the weak topology $\sigma({\cal M}_b,C_0)$.
We denote $\smes :=C([0,T];{\cal M}_b(\RR)-\sigma({\cal M}_b,C_0))$.

We recall that if a sequence of measure $(\mu_n)_{n\in \NN}$ 
in $\calM_b(\RR)$ satisfies $\sup_{n\in \NN} |\mu_n|(\RR)<+\infty$,
then we can extract a subsequence that converges for the weak topology
$\sigma({\cal M}_b,C_0)$.

% is tight, if for all $\eps>0$, there exists a compact $K\subset \RR$ such that
% for all $n$, $\mu_n(K^c)<\eps$. 
% We recall the Prokohorov theorem~:
% let $(\mu_n)_{n\in\NN}$ be a sequence in $\calM_b(\RR)$. If $(\mu_n)_{n\in\NN}$
% is tight and if $\sup_{n\in \NN} \mu_n(\RR)<+\infty$. Then we can extract
% a subsequence that converges weakly (for the topology $\sigma({\cal M}_b,C_b)$).

The coupled system \eqref{eqrhohydro}--\eqref{eqahydro}--\eqref{eqShydro} 
is interpreted in this context as a linear conservation equation
\eqref{eqrhohydro}, the velocity $b$ of which depends on the solution $S$ 
to the elliptic equation \eqref{eqShydro}, $b=a(\pa_xS)$. This actually means
that equation \eqref{eqrhohydro} is somehow nonlinear.
One convenient tool to handle such conservation equations
	\beq\label{eq.conserve}
\partial_t\rho + \pa_x(b\rho) = 0, \qquad b \mbox{ being a given function},
	\eeq
whose solutions eventually are measures in space, is the notion of duality solutions, introduced
in \cite{bj1}. 

\subsection{Linear conservation equations}\label{duality}
%%%%%%%%%%%%%%%%%%%%%%%%%%%%%%%%%%%%%%%%%%%%%%%%%%%%%%%
Duality solutions are defined as weak solutions, the test functions being Lipschitz  solutions to the backward linear transport equation
	\begin{equation}
\partial_t p + b(t,x) \partial_x p = 0, 
    \quad p(T,.) = p^T \in {\rm Lip}(\RR).
	\label{(3)}\end{equation}
A key point to ensure existence of smooth solutions to \eqref{(3)} is that the velocity field has to be compressive, in the following sense.
	\begin{definition}
We say that the function $b$ satisfies the so-called one-sided Lipschitz condition (OSL condition) if
\begin{equation}\label{OSLC}
\partial_x b(t,.)\leq \beta(t)\qquad\mbox{for $\beta\in L^1(0,T)$ in the distributional sense}.
\eeq
	\end{definition}
A formal computation shows that 
$\partial_t (p\rho) + \partial_x [b(t,x) p\rho] = 0$, and thus
	\begin{equation}
\frac{d}{dt}\left(\int_{\RR}p(t,x)\rho(t,dx)\right) = 0,
	\label{(4)}\end{equation}
which defines the duality solutions for suitable $p$'s. 
It is now quite classical that (\ref{OSLC}) ensures existence for \eqref{(3)}, but
not uniqueness, which is of great importance here to obtain stability
results and make a convenient use of \eqref{(4)}. 

Therefore, the corner stone in the construction of duality solutions is the 
introduction 
of the notion of {\it reversible} solutions to (\ref{(3)}). A complete 
statement of the definitions and properties of reversible solutions would 
be too long in the present context, so that merely a few hints are given.
Let $\mathcal L$ 
denote the set of Lipschitz continuous solutions to (\ref{(3)}), and define
the set of {\it exceptional solutions}:
	$$
{\mathcal E} = \Big\{p\in{\mathcal L} \mbox{ such that } p^T \equiv 0 \Big\}.
	$$ 
The possible loss of 
uniqueness corresponds to the case where ${\mathcal E}$ is not reduced to zero.
	\begin{definition}
We say that $p\in{\mathcal L}$ is a {\bf reversible solution} to (\ref{(3)}) 
if $p$ is locally constant on the set 
	$$
{\mathcal V}_e=\Big\{(t,x)\in [0,T] \times \RR;\ \exists\ p_e\in{\mathcal E},\ p_e(t,x)\not=0\Big\}.
	$$
	\end{definition}

This definition leads quite directly to the uniqueness results of 
\cite{bj1}. It turns out that the class of 
reversible solutions is also stable by perturbations of the coefficient $b$.

We now restrict ourselves to those $p$'s in (\ref{(4)}). More precisely,
we state the following definition.
	\begin{definition}
We say that 
$\rho\in \smes := C([0,T];{\cal M}_b(\RR)-\sigma({\cal M}_b,C_0))$
is a {\bf duality solution} to (\ref{eq.conserve}) 
if for any $0<\tau\le T$, and any {\bf reversible} solution $p$ to (\ref{(3)})
with compact support in $x$,
the function $\displaystyle t\mapsto\int_{\RR}p(t,x)\rho(t,dx)$ is constant on
$[0,\tau]$.
	\label{defdual}\end{definition}
\begin{remark}\label{dual.trans}
A similar notion of duality solution for the transport equation is available
 $\pa_t u + b\pa_x u = 0$,
and $\rho$ is a duality solution of (\ref{eq.conserve}) iff $u=\int^x\rho$
is a duality solution to transport equation (see \cite{bj1}).
\end{remark}

We shall need the following facts concerning duality solutions.

	\begin{theorem}(Bouchut, James \cite{bj1})\label{ExistDuality}
\begin{enumerate}
\item Given $\rho^\circ \in {\cal M}_b(\RR)$, under the assumptions
(\ref{OSLC}), there exists a unique $\rho \in \smes$,
duality solution to (\ref{eq.conserve}), such that $\rho(0,.)=\rho^\circ$. \\
Moreover, if $\rho^\circ$ is nonnegative, then $\rho(t,\cdot)$ is nonnegative
for a.e. $t\geq 0$. And we have the mass conservation 
	$$
|\rho(t,\cdot)|(\RR) = |\rho^\circ|(\RR), \quad \mbox{ for a.e. } t\in ]0,T[.
	$$
\item Backward flow and push-forward: the duality solution satisfies
	\beq\label{flow}
\forall\, t\in [0,T], \forall\, \phi\in C_0(\RR),\quad
\int_\RR \phi(x)\rho(t,dx) = \int_\RR \phi(X(t,0,x)) \rho^0(dx),
	\eeq
where the {\bf backward flow} $X$ is defined as the unique reversible
solution to
	$$
\pa_tX + b(t,x) \pa_xX = 0 \quad \mbox{ in } ]0,s[\times\RR, \qquad
X(s,s,x)=x.
	$$
\item For any duality solution $\rho$, we define the {\bf generalized flux}
  corresponding to $\rho$ by 
$b\Dpetit \rho = -\pa_t u$, where $u=\int^x \rho\,dx$.

There exists a bounded Borel function $\widehat b$, called {\bf
universal representative} of $b$, such that $\widehat b = a$
almost everywhere, and for any duality solution $\rho$,
	$$
\partial_t \rho + \partial_x(\widehat b\rho) = 0 \qquad \hbox{in the distributional sense.}
	$$
\item Let $(b_n)$ be a bounded sequence in
$L^\infty(]0,T[\times\RR)$, such that
$b_n\rightharpoonup b$ in $L^\infty(]0,T[\times\RR)-w\star$. Assume
$\partial_x b_n\le \alpha_n(t)$, where $(\alpha_n)$ is bounded in $L^1(]0,T[)$,
$\partial_x b\le\alpha\in L^1(]0,T[)$.
Consider a sequence $(\rho_n)\in\smes$ of duality solutions to
	$$
\partial_t\rho_n+\partial_x(b_n\rho_n)=0\quad\hbox{in}\quad]0,T[\times\RR,
	$$
such that $\rho_n(0,.)$ is bounded in ${\cal M}_{b}(\RR)$, and
$\rho_n(0,.)\rightharpoonup\rho^\circ\in{\cal M}_{b}(\RR)$.

\noindent Then $\rho_n\rightharpoonup \rho$ in $\smes$, where $\rho\in\smes$ is the
duality solution to
	$$
\partial_t\rho+\partial_x(b\rho)=0\quad\hbox{in}\quad]0,T[\times\RR,\qquad
\rho(0,.)=\rho^\circ.
	$$
Moreover, $\widehat b_n\rho_n\rightharpoonup \widehat b\rho$ weakly in ${\cal M}_{b}(]0,T[\times\RR)$.
\end{enumerate}
	\end{theorem}

The set of duality solutions is clearly a vector space, but it has
to be noted that a duality solution is not {\it a priori} defined as 
a solution in the sense of distributions. 
However, assuming that the coefficient $b$ is piecewise continuous,
we have the following equivalence result:
	\begin{theorem}\label{dual2distrib}
Let us assume that in addition to the OSL condition \eqref{OSLC}, 
$b$ is piecewise continuous on $]0,T[\times\RR$ where the 
set of discontinuity is locally finite.
Then there exists a function $\bchapo$ which coincides with $b$
on the set of continuity of $b$.

With this $\bchapo$, $\rho\in \smes$ is a duality solution to 
\eqref{eq.conserve} if and only if $\pa_t\rho+\pa_x(\bchapo\rho)=0$
in ${\mathcal D}'(\RR)$. Then the generalized flux 
$b\Dpetit \rho = \bchapo \rho$. 
In particular, $\bchapo$ is a universal representative of $b$.
	\end{theorem}

This result comes from the uniqueness of solutions to the Cauchy 
problem for both kinds of solutions (see Theorem 4.3.7 of \cite{bj1}).

\subsection{Main results}
%%%%%%%%%%%%%%%%%%%%%%%%%%%%%%%%%%%%%%%%%%%%%%%%%%%%%
We are now in position to give the definition of duality solutions for the limit system
\eqref{eqrhohydro}--\eqref{eqShydro}.
	\begin{definition}\label{defexist}
We say that $(\rho,S)\in C([0,T];\calM_b(\RR))\times C([0,T];W^{1,\infty})$ 
is a duality solution to \eqref{eqrhohydro}--\eqref{eqShydro} if there exists 
$b\in L^\infty((0,T)\times\RR)$ and $\alpha\in L^1_{loc}(0,T)$ satisfying $\pa_xb\le\alpha$ in $\cal D'$,
such that
\begin{enumerate}
\item for all $0<t_1<t_2<T$
	$$
\pa_t\rho + \pa_x(b\rho) = 0\quad\mbox{in the sense of duality on }]t_1,t_2[,
	$$
\item \eqref{eqahydro} is satisfied in the weak sense:
	$$
\forall\, \psi\in C^1(\RR), \ \forall\, t\in [0,T],
\quad \int_\RR (\pa_xS\pa_x\psi+S\psi)(t,x)\,dx = \int \psi(x)\,\rho(t,dx),
	$$
\item $\qquad b=a(\pa_xS)\quad a.e.$
\end{enumerate}
	\end{definition}

\begin{remark}
For $S$ in $C([0,T];W^{1,\infty})$ and $\phi$ as in \eqref{phi}, 
we have $a(\pa_xS)\in C([0,T];L^\infty(\RR))$. Therefore equation 
\eqref{eqrhohydro} is meaningful in the duality sense. The key 
property  is then the one-sided Lipschitz condition.
\end{remark}

Unfortunately, Definition \ref{defexist} does not ensure uniqueness, as 
we shall evidence in Section \ref{SecExistDual}. This is due
to the fact that the product $a(\pa_xS)\rho$ is not properly defined yet.
Indeed the relevant definition of this product relies on a proper definition of the flux
of the system, which we introduce now. Let $A$ be an antiderivative of $a$ such that $A(0)=0$, we set
\beq\label{DefFluxJ}
J=-\pa_x(A(\pa_xS)) +a(\pa_xS)S.
\eeq
This choice is justified first since this definition holds true when $S$ is regular. Indeed we have
$\pa_x(A(\pa_xS))=a(\pa_xS)\pa_{xx}S$, so that we can write $J=a(\pa_xS)(-\pa_{xx}S+S)=a(\pa_xS)\rho$.
On the other hand, a more physical reason relies on the fact that the above $J$ is the correct flux
for the kinetic model, and passes to the limit when $\eps$ goes to zero, see Section \ref{SecConvKinet}.

We can now establish the following uniqueness theorem:
	\begin{theorem}\label{ExistFlux}
Let us assume that $\rho^{ini}\geq 0$ is given in $\calM_b(\RR)$.
Then, for all $T> 0$ there exists
a unique duality solution $(\rho,S)$ with $\rho\geq 0$ of 
\eqref{eqrhohydro}--\eqref{eqShydro} which satisfies in the distributional sense:
\beq\label{eqrhodis}
\pa_t \rho + \pa_x J = 0,
\eeq
where $J$ is defined in \eqref{DefFluxJ}. 
It means that the universal representative in Theorem 
\ref{ExistDuality} satisfies
	$$
\widehat b \rho = J,\quad \mbox{ in the sense of measures.}
	$$
Moreover, we have $\rho = X_\# \rho^{ini}$ where $X$ is the 
backward flow corresponding to $a(\pa_xS)$.
	\end{theorem}

The second result concerns the rigorous proof of hydrodynamical limit for the kinetic model.
Let $(f_\eps, S_\eps)$ be a solution of the system \eqref{cin1D}--\eqref{ellip1D},
complemented with null boundary condition at infinity and with the following
 initial data:
	\begin{equation}\label{initcin}
f_\eps(0,\cdot,\cdot) = f_\eps^{ini},
	\end{equation}
such that $\rho_\eps^{ini}=\eta_\eps*\rho^{ini}$ where $\eta_\eps$
is a mollifier and $\rho^{ini}$ is given in $\calM_b(\RR)$.
We recall that for fixed $\eps> 0$, there exists ($f_\eps$, $S_\eps$) such that 
$f_\eps$ belongs to $C([0,T]\times \RR\times V)$
and therefore $S_\eps \in C([0,T];C^2(\RR))$, see \cite{bourncalv}, or \cite{nv} in the present context.
	\begin{theorem}\label{ConvKinet}
Let us assume that $\rho^{ini}\geq 0$ is given in $\calM_b(\RR)$.
Let $(f_\eps,S_\eps)$ be a solution to the kinetic--elliptic equation 
\eqref{cin1D}--\eqref{ellip1D} with initial data \eqref{initcin}. 
Then, as $\eps \to 0$, $(f_\eps,S_\eps)$ converges in the following
sense:
	$$
\begin{array}{c}
\ds \rho_\eps:=f_\eps(c)+f_\eps(-c) \rightharpoonup \rho
\qquad \mbox{ in } \quad 
\smes := C([0,T];{\cal M}_{b}(\RR)-\sigma({\cal M}_{b},C_0)), \\[4mm]
\ds S_\eps \rightharpoonup S \qquad \mbox{ in } \quad
C([0,T];W^{1,\infty}(\RR))-weak,
\end{array}
	$$
where $(\rho,S)$ is the unique duality solution of the system 
\eqref{eqrhohydro}--\eqref{eqShydro} satisfying
$$
\widehat b \rho = J ,\quad \mbox{ in the sense of measures.}
$$
	\end{theorem}

%%%%%%%%%%%%%%%%%%%%%%%%%%%%%%%%%%%%%%%%%%%
\section{Properties of $S$}\label{SecPropS}
%%%%%%%%%%%%%%%%%%%%%%%%%%%%%%%%%%%%%%%%%%%
We gather in this section a set of properties for the solution $S$ to \eqref{eqShydro} that will be used throughout the paper.

\subsection{One-sided estimates}
%%%%%%%%%%%%%%%%%%%%%%%%%%%%%%%%%%%%%%%
The estimates presented in this part rely only on equation \eqref{eqShydro}.
\begin{lemma}\label{propS}
Let $\rho \in C([0,T],{\cal M}_b(\RR))$. Then
the solution $S$ of equation \eqref{eqShydro} satisfies
\begin{enumerate}
\item $\rho \ge 0 \Longrightarrow S \geq 0$  
\item one-sided estimate: $\pa_{xx} S \leq S$ if and only if $\rho \ge 0$
\item for all $p\in[1,+\infty]$, $S\in C([0,T],L^p(\RR))$ and $\pa_x S\in C([0,T],L^p(\RR))$
\end{enumerate}
\end{lemma}
\begin{proof}
The first two items are easy consequences of the expression \eqref{Sexplicit} for the first one, of the equation
\eqref{eqShydro} for the second.
For the third item, from convolution properties, we have for any $p\in [1,+\infty]$
	$$
\|S(t,.)\|_{L^p(\RR)} = \frac{1}{2}\|e^{-|\cdot|} * \rho(t,.)\|_{L^p(\RR)} 
	\leq |\rho(t,.)|(\RR)\frac{1}{2}\| e^{-|\cdot|}\|_{L^p(\RR)} 
	= \frac 12\sup_{t\in[0,T]}|\rho(t,\cdot)|(\RR),
	$$
where $|\rho|(\RR)$ stands for the total mass of the 
nonnegative measure $\rho$. We proceed in the same way for $\pa_xS$.
\end{proof}

As mentioned above, the key point 
to use the duality solutions is that the velocity field 
satisfies the OSL condition \eqref{OSLC}.

\begin{lemma}\label{aOSL}
Let $\rho\in \smes$. Then the coefficient $a(\pa_xS)$ defined by 
\eqref{eqahydro}-\eqref{eqShydro} satisfies the OSL condition \eqref{OSLC}
if and only if $\rho \ge 0$
\end{lemma}
\begin{proof}
Straightforward computations lead to
$$
\pa_x (a(\pa_x S)) = - c^2 \phi'(c\pa_xS)\pa_{xx} S.
$$
With \eqref{eqShydro} and since $\phi$ is a nonincreasing function, 
we deduce from the one-sided estimate of Lemma \ref{propS}
$$
\pa_x (a(\pa_xS)) \leq \max\{c^2 \|\phi'\|_{L^\infty} S,0\}.
$$
We conclude thanks to the bound on $S$ in $L^\infty$.
\end{proof}

Finally, we turn to a convergence result for a sequence of such functions $S$.
\begin{lemma}\label{convpaS}
Let $(\rho_n)_{n\in\NN}$ be a sequence of measures that converges weakly
towards $\rho$ in $\smes$ as $n$ goes to $+\infty$. 
Let $S_n(t,x)=(K*\rho_n(t,\cdot))(x)$ and $S(t,x)=(K*\rho(t,\cdot))(x)$,
where $K$ is defined in \eqref{Sexplicit}. Then when $n\to +\infty$ we have
$$
\begin{array}{rcl}
\pa_xS_n(t,x) &\longrightarrow & \pa_xS(t,x)\quad\mbox{ for a.e. } t\in [0,T],\ x\in \RR, \\
\pa_xS_n(t,x) &\rightharpoonup & \pa_xS(t,x)\quad\mbox{ in }L^\infty_{t,x} \,weak-* .
\end{array}
$$
\end{lemma}

\begin{proof}
The proof of this result is obtained by regularization of the convolution
kernel (see Lemma 3.1 of \cite{jamesnv}).
\end{proof}

\subsection{Entropy estimates}
%%%%%%%%%%%%%%%%%%%%%%%%%%%%%%%%%%%%%%%%%%
In this subsection, we consider now that $(\rho,S)$ satisfy \eqref{eqrhodis}--\eqref{DefFluxJ}
in the sense of distributions. We prove first that $S$ satisfies a nonlinear nonlocal equation. Next, 
following the strategy of \cite{NPSm3as}, we prove that the above
one-sided estimate implies some kind of entropy inequality for $\pa_xS$.
	\begin{lemma}\label{lem:eqS}
Assume $(\rho,S)\in C([0,T];\calM_b(\RR))\times C([0,T];W^{1,\infty})$ 
satisfy \eqref{eqrhodis}--\eqref{DefFluxJ},
then $\pa_xS\in C([0,T],L^1(\RR))\cap L^\infty([0,T],BV(\RR))$ and
$S$ is a weak solution of 
	\beq\label{eqSdistrib}
\pa_t S -\pa_xK*\pa_x(A(\pa_xS)) + \pa_xK*(a(\pa_xS)S)=0.
	\eeq
	\end{lemma}

\begin{proof}
We have $\rho\in \smes$ and $\pa_{xx}S=S-\rho$. Then 
$\pa_xS\in C([0,T],L^1(\RR))\cap L^\infty([0,T],BV(\RR))$. We recall
that we have $S=K*\rho$ where $K(x)=\frac 12 e^{-|x|}$.
Thus taking the convolution by $K$ of \eqref{eqrhodis}--\eqref{DefFluxJ},
we get that $S$ is a weak solution of \eqref{eqSdistrib}.
\end{proof}

\begin{lemma}\label{entropS}
Let $S$ be a weak solution in $C([0,T];W^{1,1}(\RR))$ of
\eqref{eqSdistrib} with initial data $S^{ini}$.
We assume moreover that $\pa_xS$ belongs to $L^\infty([0,T];BV(\RR))$ and
that the one-sided estimate $\pa_{xx} S \leq S$ holds in the distributional
sense. 
Then for any twice continuously differentiable convex function $\eta$ 
we have 
\begin{equation}
  \label{eq:entropy}
  \pa_t \eta(\pa_xS)+ \pa_x(q(\pa_xS)) -\eta'(\pa_xS)a(\pa_xS)S + \eta'(\pa_xS)[K*(-\pa_xA(\pa_xS) + a(\pa_xS)S)] \leq 0, 
\end{equation}
where the entropy flux $q$ is defined by
$$
q(x) = \int_0^x \eta'(y) a(y)\,dy.
$$
\end{lemma}

\begin{proof} 
From Lemma \ref{lem:eqS}, $S$ satisfies \eqref{eqSdistrib}. By differentiation, and using
 the property $\pa_{xx}K=K-\delta_0$, we get
\beq\label{eqdxSdistrib}
\pa_t\pa_xS+\pa_xA(\pa_xS) -a(\pa_xS)S + K*(-\pa_xA(\pa_xS) + a(\pa_xS)S)=0.
\eeq
Consider a sequence of mollifiers $\zeta_n(x)=n\zeta(nx)$, with
$n\in \NN$, $\zeta\in C_0^\infty(\RR)$, $\zeta\geq 0$ and 
$\int_\RR\zeta(x)\,dx=1$. We set $S_n=\zeta_n*S$. 
Then we have
$$
\pa_t\pa_xS_n+\pa_x(A(\pa_xS)*\zeta_n) -\zeta_n*(a(\pa_xS)S) + 
K*\zeta_n*(-\pa_xA(\pa_xS) + a(\pa_xS)S)=0.
$$
We define the commutators $R_n$ and $Q_n$ as follows:
	$$
A(\pa_xS)*\zeta_n = A(\pa_xS_n)+ R_n(t,x),
	$$
	$$\begin{array}{rcl}
\ds Q_n(t,x) &=& \ds -\zeta_n*(a(\pa_xS)S) + 
K*\zeta_n*(-\pa_xA(\pa_xS) + a(\pa_xS)S)  \\[2mm]
	&&\qquad {}+ a(\pa_xS_n)S_n - K*(-\pa_xA(\pa_xS_n)+ a(\pa_xS_n)S_n),
	\end{array}$$
so that the regularized solution satisfies
\begin{equation}
  \label{eq:dxSn}
  \pa_t\pa_xS_n + \pa_x(A(\pa_xS_n)+R_n)-a(\pa_xS_n)S_n + K*(-\pa_xA(\pa_xS_n) + a(\pa_xS_n)S_n)+Q_n=0.
\end{equation}
Let us consider $\eta$ a twice continuously differentiable convex function
and let $q$ be the corresponding entropy flux.
Multiplying equation \eqref{eq:dxSn} by $\eta'(\pa_xS_n)$, we get
\begin{equation}
  \label{eq:Snentrop}
  \pa_t \eta(\pa_xS_n) + \pa_x(q(\pa_xS_n)+\eta'(\pa_xS_n)R_n) + H_n =-\eta'(\pa_xS_n)Q_n+R_n\pa_x(\eta'(\pa_xS_n)),
\end{equation}
where 
	$$
H_n :=  -\eta'(\pa_xS_n)a(\pa_xS_n)S_n + \eta'(\pa_xS_n)[K*(-\pa_xA(\pa_xS_n) + a(\pa_xS_n)S_n)].
	$$

Due to properties of the convolution product, we have
	$$
  R_n\to 0, \qquad Q_n\to 0 \qquad \mbox{ in } L^p_{loc}((0,\infty)\times\RR),
\quad 1\leq p<+\infty,
	$$
so that in the sense of distribution, we have straightforwardly
$$
\pa_x(\eta'(\pa_xS_n)R_n) \to 0, \qquad \eta'(\pa_xS_n) Q_n \to 0
$$
and
$$
H_n \to H:=-\eta'(\pa_xS)a(\pa_xS)S + \eta'(\pa_xS)[K*(-\pa_xA(\pa_xS) + a(\pa_xS)S)],
$$
which is precisely the desired term in the limit equation.
Now we deal with the term $R_n\pa_x(\eta'(\pa_xS_n))$ on the right-hand side, and 
we notice that $R_n\geq 0$ thanks to the Jensen inequality and the convexity of $A$.
Therefore, since $\eta$ is convex, we have
	$$
R_n\pa_x(\eta'(\pa_xS_n)) = R_n \eta''(\pa_xS_n) \pa_{xx}S_n \leq 
R_n \eta''(\pa_xS_n) S_n,
	$$
where we have used the one-sided estimate $\pa_{xx}S_n\leq S_n$ to obtain
the last inequality.
Since $S_n$ is bounded in $L^\infty$ independently of $n$, we can
pass to the limit in this last identity thanks to the Lebesgue
dominated convergence theorem to get
$$
R_n \eta''(\pa_xS_n) S_n \to 0 \quad \mbox{ in } L^1_{loc}((0,\infty)\times \RR).
$$
Finally, letting $n$ going to $+\infty$ in \eqref{eq:Snentrop}, 
we deduce that \eqref{eq:entropy} holds in the distributional sense.
\end{proof}

\begin{remark}
This equation relies strongly on the definition of the flux $J$ in \eqref{DefFluxJ}.
This fact has already been noticed by the authors in \cite{note},
which can be viewed as a particular case of the one studied in this paper
by replacing the elliptic equation \eqref{eqShydro} for $S$ 
by the Poisson equation $-\pa_{xx}S=\rho$.
In this case, the product of $a(\pa_xS)$ by $\rho$ is naturally defined
by $a(\pa_xS)\rho=-\pa_xA(\pa_xS)$, so that equation on $S$ corresponding 
to \eqref{eqdxSdistrib} is given by
	$$
\pa_t\pa_xS + \pa_x A(\pa_xS) = 0.
	$$
This equation is a nonlinear hyperbolic conservation law which is local, 
contrary to \eqref{eqdxSdistrib}. Therefore uniqueness is ensured by entropy conditions.
Since $\pa_xS$ is monotonous ($-\pa_{xx}S=\rho \geq 0$), this can be formulated as 
a chord condition on $A$ (see \cite{BJpg}). If in addition $A$ is convex or concave (i.e. if $a$ is non-decreasing
or non-increasing), this selects only increasing or decreasing shocks.
\end{remark}

%%%%%%%%%%%%%%%%%%%%%%%%%%%%%%%%%%%%%%%%%%%%%%%%%%%%%%%%%%%%%%%%%%%%%%%%%%%%%%%%%%%%%%%%%%%%%%
\section{Existence and uniqueness for the hydrodynamical problem}
\label{SecExistDual}
%%%%%%%%%%%%%%%%%%%%%%%%%%%%%%%%%%%%%%%%%%%%%%%%%%%%%%%%%%%%%%%%%%%%%%%%%%%%%%%%%%%%%%%%%%%%%%
In this Section, we focus on the proof of Theorem \ref{ExistFlux}, which can be split in 3 steps. The first one consists in
obtaining the dynamics of aggregates, or in other words of combinations of Dirac masses. Next we obtain the existence of
duality solutions in the sense of Definition \ref{defexist} by proving first that aggregates 
define such a solution, then proceeding to the general case by approximation. This is exactly the same strategy as 
for the pressureless gases in \cite{BJpg}. Finally, uniqueness follows from a careful definition
of the flux of the equation. In this respect, we first underline with an example that Definition \ref{defexist} as it stands does not
give uniqueness, and how the proper definition of the flux singles out a unique solution.

Indeed, let us consider \eqref{eqrhohydro}--\eqref{eqShydro} with 
boundary condition \eqref{bordhydro} where the initial datum 
is assumed to be a Dirac mass in $0$: $\rho^{ini}=\delta_0$. 
We have that $(\delta_0, K*\delta_0)$ is a solution to 
\eqref{eqrhohydro}--\eqref{eqShydro} with initial data $\delta_0$.
Actually, the pair
	\beq\label{rhononuniq}
\rho_1(t,x)=\delta_{x_1(t)}(x);\quad S_1(t,x)=K*\rho_1(t,x)=\frac 12 e^{-|x-x_1(t)|}.
	\eeq
turns out to define a solution in the sense of duality in Definition \ref{defexist}
for several choices of curves $x_1$ with $x_1(0)=0$. Set $b_1(t,x)=a(\pa_xS_1)(t,x)$, and notice 
first that, according to Remark \ref{dual.trans},
$\rho_1$ is a duality solution if $u_1:=\int^x \rho_1\,dx=H(x-x_1(t))$
is a duality solution of the transport equation. Now, from Lemma \ref{aOSL},
$b_1$ satisfies the OSL condition, therefore $u_1$ is a duality solution 
of the transport equation as soon as it is solution in the sense of distributions. 
As detailed in \cite{bj1}, Section 3,
this holds true only if $u$ satisfies some admissibility conditions, namely, the characteristics
of the velocity field have to enter the discontinuity on both side. 
Since $\lim_{x\to x_1^+} b_1(x)=a(-1/2)$ and $\lim_{x\to x_1^-} b_1(x)=a(1/2)$,
the velocity of the shock should satisfy
$a(1/2)>x'_1(t)>a(-1/2)$, which furnishes an infinity of solution.

For any of the previous solutions, the generalized flux given by Theorem \ref{ExistDuality}--3
is $b_1\Dpetit \rho_1 = -\pa_t u_1=-x'_1(t)\delta_{x_1(t)}$.
On the other hand, let us compute the flux $J$ defined by \eqref{DefFluxJ}.
For simplicity, we set here $\alpha=0$ in the definition \eqref{phi} of $\Phi$. With this convention, we get
	$$
a(\pa_xS_1)(t,x)=\left\{
\begin{array}{ll}
\ds -\lambda c, \ & x<x_1(t), \\[2mm]
\ds \lambda c, & x>x_1(t),
\end{array}
\right.
\quad A(\pa_xS_1)(t,x)=\frac 12 \left\{
\begin{array}{ll}
\ds -\lambda c e^{x-x_1(t)}, \quad & x<x_1(t), \\[2mm]
\ds -\lambda c e^{-x+x_1(t)}, & x>x_1(t).
\end{array}
\right.
	$$
Obviously we have $J=0$, so that the condition $\achapo \rho=J$
selects $x'_1(t)=0$, which finally implies $x_1\equiv 0$ since $x_1(0)=0$.

\subsection{Dynamics of aggregates}\label{aggregat}
%%%%%%%%%%%%%%%%%%%%%%%%%%%%%%%%%%%%%%%%%%%%%%%%%%%%%%%%%%

Let us first consider the motion of aggregates. 
We assume that $\rho^{ini}_n$ is given by a finite sum of Dirac masses:
$\rho^{ini}_n=\sum_{i=1}^n m_i \delta_{x_i^0}$ where $x_1^0<x_2^0<\dots<x_n^0$
and the $m_i$-s are nonnegative.
We look for a couple $(\rho_n,S_n)$ solving in the distributional sense
$\pa_t\rho_n +\pa_x J_n=0$ where the flux $J_n$ is given by \eqref{DefFluxJ} and 
$S_n$ solves \eqref{eqShydro}. We recall that it means that
$S_n=K*\rho_n$ where $K$ is defined in \eqref{Sexplicit}.
Let us set $\rho_n(t,x)=\sum_{i=1}^n m_i \delta_{x_i(t)}$. 
Such a function is a solution in the sense of distributions of
\eqref{eqrhodis} if the function $u_n$ defined by
	\beq\label{defu}
u_n(t,x):=\int^x\rho_n\,dx = \sum_{i=1}^n m_i H(x-x_i(t)),
	\eeq
where $H$ denotes the Heaviside function, is a distributional solution to
	\beq\label{equdis}
\pa_t u_n -\pa_xA(\pa_xS_n) + a(\pa_xS_n) S_n =0.
	\eeq
We have
	$$
S_n(t,x)=\sum_{i=1}^n \frac{m_i}{2} e^{-|x-x_i(t)|},
	$$
	\beq\label{dxSn}
\pa_xS_n(t,x)=-\sum_{i=1}^n \frac{m_i}{2}\mbox{ sign }(x-x_i(t)) e^{-|x-x_i(t)|}.
	\eeq
Straightforward computations prove that we have in the distributional sense
\beq\label{calAdis}
\pa_x A(\pa_xS_n)=a(\pa_xS_n) S_n + \sum_{i=1}^n [A(\pa_xS_n)]_{x_i} \delta_{x_i},
\eeq
where $[f]_{x_i}=f(x_i^+)-f(x_i^-)$ is the jump of the function $f$ at $x_i$.
Injecting \eqref{defu} and \eqref{calAdis} in \eqref{equdis}, we find
$$
-\sum_{i=1}^n m_i x'_i(t) \delta_{x_i(t)} = \sum_{i=1}^n [A(\pa_xS_n)]_{x_i} \delta_{x_i}.
$$
Thus the dynamics of aggregates is finally given by
$$
m_i x'_i(t) = -[A(\pa_xS_n)]_{x_i(t)}, \quad \mbox{ for } i=1,\dots, n.
$$
We complement this system of ODEs by the initial data $x_i(0)=x_i^0$.
More precisely, recalling that $K(x)=\frac{1}{2}e^{-|x|}$, using \eqref{dxSn}
this latter system can be rewritten~:
\beq\label{dynagg}
m_i x'_i(t)=A\left(\frac{m_i}{2}+ \sum_{j\neq i} m_j \pa_xK(x_j-x_i)\right)
-A\left(-\frac{m_i}{2}+ \sum_{j\neq i} m_j \pa_xK(x_j-x_i)\right).
\eeq
Recall that, from the definition of the coefficient $a$ in \eqref{eqahydro} with
\eqref{phi}, $a$ is nondecreasing and odd, so that $A$ is a convex function. This 
implies that for $i=1,\dots,n-1$, $x'_i\geq x'_{i+1}$, therefore, aggregates can collapse in finite time but 
an aggregate cannot split. This is a direct consequence of the fact that 
we are considering positive chemotaxis, i.e. $a$ is nondecreasing.
If there exists a time $t_1$ for which we have for instance 
$x_i(t_1)=x_{i+1}(t_1)$, then the dynamics for $t>t_1$ is defined 
as above except that we replace $m_i$ by $m_i+m_{i+1}$ and 
$x_i(t)=x_{i+1}(t)$ for $t>t_1$.
Moreover $A$ is even, then when $n=1$, we have $x'_1=0$ and $x_1(t)=x_1^0$.
Thus if aggregates collapse such that they form a single aggregate of mass 
$\sum_i m_i$, then this aggregate does not move for larger times.

\subsection{Existence of duality solutions}\label{sub:existDual}
%%%%%%%%%%%%%%%%%%%%%%%%%%%%%%%%%%%%%%%%%%%%%%%%%%%%%%ù

We have constructed $(\rho_n,S_n)$ which is a solution of
\eqref{eqrhodis}-\eqref{DefFluxJ}-\eqref{eqShydro} in the distributional sense
for the given initial data $\rho_n^{ini}$.
We recall the following result due to Vol'pert \cite{volpert}
(see also \cite{ambrosioBV}): if $u$ belongs to $BV(\RR)$ and
$f\in C^1(\RR)$ with $f(0)=0$, then $v=f\circ u$ belongs to $BV(\RR)$
and 
	$$
\exists\, \bar{f_u}\ \mbox{ with } \bar{f_u}=f'(u) \mbox{ a.e. }\ 
\mbox{ such that }\ (f\circ u)' = \bar{f_u} u'.
	$$
Together with the fact that $A$ is an antiderivative of $a$ such that $A(0)=0$,
 this result implies that there exists a function $\achapo_n$ such that
$$
J_n:=-\pa_x(A(\pa_xS_n))+a(\pa_xS_n)S_n = \achapo_n \rho_n, 
\quad \mbox{ and } \quad \achapo_n=a(\pa_xS_n) \mbox{ a.e. }
$$
Thus $\rho_n$ is a solution in the distributional sense of
$$
\pa_t\rho_n + \pa_x(\achapo_n \rho_n) = 0.
$$
Moreover, we deduce from \eqref{dxSn} that $a(\pa_xS_n)$ is piecewise
continuous with the discontinuity lines defined by $x=x_i$, $i=1,\dots,n$.
We can apply Theorem \ref{dual2distrib} which gives that 
$\rho_n$ is a duality solution and that $\achapo_n$ is a universal
representative of $a(\pa_xS_n)$. Then the flux is given by
$a(\pa_xS_n)\Dpetit \rho_n=J_n$.

Let us yet consider the case of any initial data $\rho^{ini}\in \calM_b(\RR)$.
We approximate $\rho^{ini}$ by $\rho^{ini}_n=\sum_{i=1}^n m_i \delta_{x_i^0}$ 
with $\rho_n^{ini}\rightharpoonup \rho^{ini}$ in $\calM_b(\RR)$.
By the same token as above, we can construct a solution 
$(\rho_n,S_n=K*\rho_n)$ with $\rho_n(t=0)=\rho_n^{ini}=\sum_{i=1}^n m_i\delta_{x_i^0}$, 
which solves in the sense of duality
	$$
\pa_t\rho_n +\pa_x(a(\pa_xS_n)\rho_n)= 0,
	$$
in the sense of distributions
	$$
\pa_t\rho_n + \pa_x J_n=0, \quad J_n = -\pa_x A(\pa_xS_n)+a(\pa_xS_n) S_n,
	$$
 and which satisfies
	$$
\achapo_n \rho_n = J_n, \quad \achapo_n=a(\pa_xS_n) \mbox{ a.e. }
	$$
Moreover, since $\pa_xS_n$ is bounded in $L^\infty$ uniformly with 
respect to $n$ by construction,
we can extract a subsequence of $(a(\pa_xS_n))_n$ that converges
in $L^\infty -weak*$ towards $b$.
Since from Lemma \ref{aOSL}, $a(\pa_xS_n)$ satisfies the OSL condition, 
we deduce from Theorem \ref{ExistDuality} 4) that, up to an extraction,
$\rho_n\rightharpoonup \rho$ in $\smes$ and $\achapo_n\rho_n\rightharpoonup \achapo \rho$
weakly in $\calM_b(]0,T[\times\RR)$, $\rho$ being a duality solution
of the scalar conservation law with coefficient $b$.
With Lemma \ref{convpaS}, we deduce that $\pa_xS_n\to \pa_xS$ a.e., 
it implies in particular that 
$J_n\to J:=-\pa_xA(\pa_xS)+a(\pa_xS)S$ in ${\mathcal D}'(\RR)$ 
and that $a(\pa_xS_n)\to a(\pa_xS)$ a.e. By uniqueness of the weak
limit, we have $b=a(\pa_xS)$. Moreover $J=\achapo \rho$ a.e. and
$\rho$ satisfies then \eqref{eqrhodis}.
Then $(\rho,S)$ is a solution as in Theorem \ref{ExistFlux}, 
this concludes the proof of the existence.

\subsection{Uniqueness of solutions}\label{uniqueS}
%%%%%%%%%%%%%%%%%%%%%%%%%%%%%%%%%%%%%%%%%%%%%%%%%%%%%%%

Let us consider yet the study of the uniqueness.
As shown above, Definition \ref{defexist} is not sufficient
to ensure uniqueness. Therefore, we will use the fact that we have 
a duality solution $\rho$ that satisfies \eqref{eqrhodis}
in ${\cal D}'([0,T]\times \RR)$ with the initial data $\rho^{ini}$
and with the flux $J$ given by \eqref{DefFluxJ}.
This equation leads to the non-local evolution equation on $S$ 
\eqref{eqSdistrib} as stated in Lemma \ref{lem:eqS}.

Another key point is the one-sided estimate $\pa_{xx}S\leq S$. In fact,
if we consider for instance $\rho^{ini}=0$, then it is obvious that
$\rho=0$ is a solution of \eqref{eqrhohydro}--\eqref{eqShydro}. 
However, if we allow $\rho$ to be nonpositive, i.e. if the corresponding
chemoattractant concentration $S$ does not satisfy the one-sided estimate
$\pa_{xx}S\leq S$, then we can build a simple example of non-uniqueness.
Indeed we have that
	$$
\rho(t,x)=\delta_{-x_1(t)}(x) -2 \delta_0(x) +\delta_{x_1(t)}(x)
	$$
is a duality solution of \eqref{eqrhohydro}--\eqref{eqShydro} which 
satisfies \eqref{eqrhodis}, provided $x_1(0)=0$ and 
\eqref{dynagg} is satisfied. This readily gives
	$$
x'_1(t)=A\big(\frac 12 +e^{-x_1}+\frac 12 e^{-2x_1}\big)
-A\big(-\frac 12 + e^{-x_1} +\frac 12 e^{-2x_1}\big).
	$$
Here by convexity of $A$, we have $x'_1\geq 0$.

\begin{theorem}\label{uniqS}
Let $S_1$ and $S_2$ be two weak solutions in $C([0,T];W^{1,1}(\RR))$ of
\eqref{eqSdistrib} with initial data $S_1^{ini}$ and 
$S_2^{ini}$ respectively. If we assume moreover that
$\pa_xS_1$ and $\pa_xS_2$ belongs to $L^\infty([0,T];BV(\RR))$ and 
that the one-sided estimate
$$
\pa_{xx}S_i \leq S_i, \qquad i=1,2,
$$
holds in the distributional sense.
Then there exists a nonnegative constant $C$ such that 
$$
\|S_1-S_2\|_{L^\infty([0,T];W^{1,1}(\RR))} \leq 
C \|S_1^{ini}-S_2^{ini}\|_{W^{1,1}(\RR)}.
$$
\end{theorem}

\begin{proof}
We start from the entropy inequality \eqref{eq:entropy} of Lemma 
\ref{entropS}. 
Using standard regularization arguments, it is well-known that we can 
apply this inequality to the family of Kru\v{z}kov entropies $\eta_\kappa(u) = |u-\kappa|$. 
Then, the doubling of variables technique developed by Kru\v{z}kov allows to 
justify the following computation. 
Assume $S_1$ and $S_2$ are two weak solutions of \eqref{eqSdistrib}, 
then in the distributional sense, we have
$$
\begin{array}{c}
\ds \pa_t|\pa_x(S_1-S_2)| + \pa_x(\mbox{ sign}(\pa_xS_1-\pa_xS_2)(A(\pa_xS_1)-A(\pa_xS_2))) \leq  \\[2mm]
\ds \mbox{ sign}(\pa_xS_1-\pa_xS_2)\big(\pa_xK*(A(\pa_xS_1)-A(\pa_xS_2)) + 
a_1S_1-a_2S_2 - K*(a_1S_1-a_2S_2)\big),
\end{array}
$$
where we denote $a_1=a(\pa_x S_1)$ and $a_2=a(\pa_xS_2)$.
Integrating with respect to $x$ and using the properties
of the convolution product, we deduce
$$
\frac{d}{dt} \int_\RR |\pa_x(S_1-S_2)|\,dx \leq \|\pa_xK\|_{\infty}
\int_\RR |A(\pa_xS_1)-A(\pa_xS_2)|\,dx + (1+\|K\|_{\infty}) 
\int_\RR |a_1S_1 - a_2S_2|\,dx.
$$
The function $a$ being regular, we have
\beq\label{borndxS12}
\frac{d}{dt} \int_\RR |\pa_x(S_1-S_2)|\,dx \leq C_0
\int_\RR |\pa_x(S_1-S_2)|\,dx + C_1 \int_\RR |S_1-S_2|\,dx.
\eeq
In the same way as for equation \eqref{eqSdistrib}, this leads to 
\beq\label{bornS12}
\frac{d}{dt} \int_\RR |S_1-S_2|\,dx \leq C_2
\int_\RR |\pa_x(S_1-S_2)|\,dx + C_3 \int_\RR |S_1-S_2|\,dx.
\eeq
Summing \eqref{bornS12} and \eqref{borndxS12}, we deduce that there 
exists a nonnegative constant $C$ such that
$$
\frac{d}{dt} \|S_1-S_2\|_{W^{1,1}(\RR)} \leq C\|S_1-S_2\|_{W^{1,1}(\RR)}.
$$
Applying the Gronwall Lemma allows to conclude the proof.
\end{proof}

{\bf Proof of uniqueness in Theorem \ref{ExistFlux}.}
Let us assume that we have two duality solutions $(\rho_1,S_1)$ and 
$(\rho_2,S_2)$ such as in Theorem \ref{ExistFlux}. Therefore,
from Lemma \ref{lem:eqS}, $S_1$ and $S_2$ are weak solutions of 
\eqref{eqSdistrib}. Using Theorem \ref{uniqS}, we conclude that $S_1=S_2$.
Thus $\rho_1=K*S_1=K*S_2=\rho_2$.
\qed

%%%%%%%%%%%%%%%%%%%%%%%%%%%%%%%%%%%%%%%%%%%%%%%%%%%%%%%%%%%%%%%%%%%%%%
\section{Convergence for the kinetic model}\label{SecConvKinet}
%%%%%%%%%%%%%%%%%%%%%%%%%%%%%%%%%%%%%%%%%%%%%%%%%%%%%%%%%%%%%%%%%%%%%%

In this section we investigate the rigorous derivation of 
\eqref{eqrhohydro}--\eqref{eqShydro} from the microscopic model \eqref{cin1D}.
First we state some estimates on the moments of the solution of the kinetic 
problem.
\begin{lemma}\label{momentf}
Let $(f_\eps,S_\eps)$ be a solution of the kinetic problem 
\eqref{cin1D}--\eqref{ellip1D}. Then for all $t\in[0,T]$ and
all $\eps>0$ we have
	$$
\int_\RR\int_V |v|^k f_\eps \,dxdv = |v|^k |\rho^{ini}|(\RR)\, ,\quad 
\mbox{ k }\in \NN. 
	$$
\end{lemma}
\begin{proof}
Since $v\in V=\{-c,c\}$, $|v|$ is constant therefore 
	$$
\int_\RR\int_V |v|^k f_\eps\,dxdv = |v|^k \int_\RR \rho_\eps\,dx.
	$$
The result follows then directly from the mass conservation in \eqref{cin1D}.
\end{proof}

\noindent{\bf Proof of Theorem \ref{ConvKinet}. }
Let $(f_\eps,S_\eps)$ be a solution of \eqref{cin1D}--\eqref{ellip1D}.
For fixed $\eps>0$, we have $f_\eps \in C([0,T]\times\RR\times V)$.
Define $\rho_\eps:=\int_V f_\eps\,dv$, $J_\eps:=\int_V v f_\eps\,dv$ and 
$
a(\pa_xS_\eps)= - c \phi(c\pa_xS_\eps).
$
We can rewrite the kinetic equation \eqref{cin1D} as
$$
\pa_t f_\eps + v\pa_xf_\eps= \frac{1}{\eps}(\Phi(-v\pa_xS_\eps)\rho_\eps
-2 f_\eps). 
$$
Taking the zeroth and first order moments, we get
\begin{eqnarray}
&\ds \pa_t \rho_\eps + \pa_x J_\eps = 0,
\label{moment0} \\[2mm]
&\ds \pa_t J_\eps + v^2 \pa_x \rho_\eps = \frac{2}{\eps}
 (a(\pa_xS_\eps)\rho_\eps -J_\eps).
\label{moment1}
\end{eqnarray}

From \eqref{moment0}, we deduce that $\forall\, t\in [0,T]$, $|\rho_\eps(t,\cdot)|(\RR) = |\rho^{ini}|(\RR)$. 
Therefore, for all $t\in [0,T]$ the sequence $(\rho_\eps(t,\cdot))_\eps$
is relatively compact in $\calM_b(\RR)-\sigma(\calM_b(\RR),C_0(\RR))$.
Moreover, there exists $u_\eps\in L^\infty([0,T],BV(\RR))$ such that
$\rho_\eps=\pa_x u_\eps$. From \eqref{moment0}, we get that 
$\pa_tu_\eps=-J_\eps$ and thanks to Lemma \ref{momentf} we deduce that
$u_\eps$ is bounded in Lip$([0,T],L^1(\RR))$. This implies the equicontinuity 
in $t$ of $(\rho_\eps)_\eps$. Thus the sequence $(\rho_\eps)_\eps$ is 
relatively compact in $\smes$ and we can extract a subsequence still
denoted $(\rho_\eps)_\eps$ that converges towards $\rho$ in $\smes$.

We recall that $S_\eps(t,x)=(K*\rho_\eps(t,\cdot))(x)$ where $K(x)=\frac 12 e^{-|x|}$.
Denoting $S(t,x):=(K*\rho(t,\cdot))(x)$, since $\rho\in \smes$, we have
$\pa_xS\in L^\infty([0,T];BV(\RR))$. From Lemma \ref{convpaS}, 
the sequence $(\pa_xS_\eps)_\eps$ converges in $L^\infty w-*$ and a.e. to $\pa_xS$ as $\eps$ goes to $0$. Lemma 
\ref{aOSL} ensures that both $a(\pa_xS_\eps)$ and $a(\pa_xS)$ satisfy the OSL condition.

From \eqref{moment0}--\eqref{moment1}, we have in the distributional sense
\beq\label{eqrhoeps1}
\pa_t\rho_\eps+\pa_x(a(\pa_xS_\eps)\rho_\eps) = 
\pa_x(a(\pa_xS_\eps)\rho_\eps -J_\eps)=
\frac {\eps}{2} \pa_x (\pa_tJ_\eps + v^2\pa_x\rho_\eps) = R_\eps.
\eeq
Now, for all $\psi\in C^2_c((0,T)\times\RR)$, we deduce from Lemma \ref{momentf}
$$
\left|\int (\pa_tJ_\eps + v^2 \pa_x\rho_\eps)\pa_x\psi\,dxdt \right| 
\leq |v||\rho^{ini}|(\RR) \|\pa_t\pa_x\psi\|_{L^\infty} + 
|v|^2|\rho^{ini}|(\RR)\|\pa_{xx}\psi\|_{L^\infty}.
$$
This implies that the limit in the distributional sense of the right-hand 
side $R_\eps$ of \eqref{eqrhoeps1} vanishes.

Now we multiply equation \eqref{ellip1D} by $a(\pa_xS_\eps)$
and use again the antiderivative $A$ of $a$ to obtain
\beq\label{fluxeps}
a(\pa_xS_\eps)\rho_\eps = -\pa_x(A(\pa_xS_\eps)) + a(\pa_xS_\eps)S_\eps,
\eeq
so that we can rewrite the conservation equation \eqref{eqrhoeps1} as follows, in ${\cal D}'(\RR)$:
\beq\label{eqrhoeps}
\pa_t\rho_\eps +\pa_x\left(-\pa_x A(\pa_xS_\eps) + a(\pa_xS_\eps)S_\eps\right)
=\frac {\eps}{2} \pa_x (\pa_tJ_\eps + v^2\pa_x\rho_\eps).
\eeq
Taking the limit $\eps\to 0$ of equation \eqref{eqrhoeps} in the sense of distributions, 
we get
	\beq\label{eq:rholim}
\pa_t\rho +\pa_x\left(-\pa_x A(\pa_xS) + a(\pa_xS)S\right) = 0\qquad\mbox{in }{\cal D}'(\RR),
	\eeq
where $S(t,x)=(K*\rho(t,\cdot))(x)$.
Therefore the pair $(\rho,S)$ satisfies \eqref{eqrhodis}--\eqref{DefFluxJ}. 
In addition, $\rho$ is nonnegative as a limit of nonnegative measures, so that Lemma \ref{propS} implies the 
one-sided estimate $\pa_{xx}S\le S$. Thus we are in position to apply Lemma \ref{lem:eqS} and Theorem \ref{uniqS}, which
give uniqueness for $S$, and consequently for $\rho$. Therefore the whole sequence $\rho_\eps$ converges
to $\rho$ in $\smes$.
We recall that we have chosen the initial data such that 
$\rho_\eps^{ini}=\eta_\eps*\rho^{ini}$ where $\eta_\eps$ is a mollifier.
Therefore $\rho_\eps^{ini}\rightharpoonup \rho^{ini}$ in 
$\calM_b(\RR)-\sigma(\calM_b(\RR),C_0(\RR))$. 

Thus we have constructed a solution that satisfies \eqref{eq:rholim}
in the distributional sense, in other words, we have defined a solution 
of the problem \eqref{eqrhohydro}--\eqref{eqShydro} thanks to its flux.
A natural question is to know whether we can define a velocity 
corresponding to this flux. From the theory of duality solutions
(see Theorem \ref{ExistDuality}), it boils down to show that the 
above constructed solution is a duality solution. 
From Vol'pert calculus \cite{volpert} we infer
the existence of $a_S$ such that $a_S=a(\pa_xS)$ a.e. and 
	$$
\pa_x(A(\pa_xS)) = a_S \pa_{xx}S.
	$$
Therefore
\beq\label{flux_Volpert}
-\pa_x(A(\pa_xS)) + a(\pa_xS)S = a_S \rho \ \mbox{ a.e. }, 
\quad \mbox{with }\ a_S = a(\pa_xS) \ \mbox{ a.e. }
\eeq
Using equation \eqref{eq:rholim} we have in the distributional sense
\beq\label{eqbis}
\pa_t \rho + \pa_x (a_S \rho) = 0.
\eeq
However, we have proved in Section \ref{uniqueS} that such a solution
is unique. We deduce that the solution $(\rho,S)$ obtained by the 
hydrodynamical limit above is the duality solution of Theorem \ref{ExistFlux}.
It concludes the proof of Theorem \ref{ConvKinet}.
\qed

\begin{remark}
In the proof above, the macroscopic flux $J$ defined in \eqref{DefFluxJ} 
appears to be the limit of the microscopic flux $J_\eps$.
Indeed from \eqref{moment1} and \eqref{fluxeps} we
deduce that, in the distributional sense,
	$$
J_\epsilon \longrightarrow J := -\pa_xA(\pa_xS)+a(\pa_xS)S.
	$$
This natural definition of the flux allows to get the uniqueness 
of the solutions of the coupled system \eqref{eqrhohydro}--\eqref{eqShydro}
thanks to equations \eqref{eqSdistrib}--\eqref{eqdxSdistrib}.
Such a technique to establish the hydrodynamic limit has been proposed
in \cite{NPS}. But the authors do not state that their limit is a duality
solution and do not define a velocity and therefore a flow corresponding 
to their flux. In the limit of the Vlasov-Poisson-Fokker-Planck system, 
this result has been investigated in \cite{note}.
\end{remark}

%%%%%%%%%%%%%%%%%%%%%%%%%%%%%%%%%%%%%%%%%%%%%%%%%%%%%%%%%%%%%%%%%%%%%%%%%%%%
\section{Numerical issue}
\label{Numeriq}
%%%%%%%%%%%%%%%%%%%%%%%%%%%%%%%%%%%%%%%%%%%%%%%%%%%%%%%%%%%%%%%%%%%%%%%%%%%%

\subsection{Finite time of collapse}

Before focusing on the numerical simulations, let us clarify the 
dynamics of the model.
In the case of $n$ Dirac masses, $m_i\geq 0$ for $i=1,\ldots,n$,
located at positions $x_1<\dots<x_n$, we recall that the
time evolution is governed by system \eqref{dynagg}:
\beq\label{dynagg1}
m_i x'_i(t) = A\left(\frac{m_i}{2}+ \sum_{j\neq i} m_j \pa_xK(x_j-x_i)\right)
-A\left(-\frac{m_i}{2}+ \sum_{j\neq i} m_j \pa_xK(x_j-x_i)\right),
\eeq
for $i=1,\dots,n$, 
where we recall that $A$ is an antiderivative of $a$ such that $A(0)=0$.
We deduce that for all $t>0$, and for $i=1,\dots,n$,
\beq\label{dynagg2}
\begin{array}{l}
\ds \exists\, \gamma_i\in \left(-\frac{m_i}{2}+ \sum_{j\neq i} m_j \pa_xK(x_j-x_i),
\frac{m_i}{2}+ \sum_{j\neq i} m_j \pa_xK(x_j-x_i)\right) \\[2mm]
\ds \mbox{ such that }
x'_i(t)=a(\gamma_i(t)).
\end{array}
\eeq

\begin{proposition}\label{propdynamic}
Let us assume that there exists $n\in \NN^*$ such that
	$$
\rho^{ini}(x)= \sum_{i=1}^n m_i^0 \delta_{x_i^0}(x),
	$$
with $m_i^0\geq 0$, for $i=1,\ldots,n$.
We assume in addition that $a$ is a nondecreasing and odd real function.
Then the duality solution $\rho$ of Theorem \ref{ExistFlux} 
has the following properties~:
\begin{enumerate}
\item If $n=1$, $x_1(t)=x_1^0$ for all $t>0$. Then $\rho(t)=\rho^{ini}$ 
for all $t>0$.
\item For $i=1,\dots,n-1$, $x'_i(t)\geq x'_{i+1}(t)$ therefore 
$x_{i+1}-x_i\leq x_{i+1}^0-x_i^0$.
\item There exists $c^*\in [x_1^0,x_n^0]$ and $T^*>0$ such that 
$\rho(t,x)=\delta_{c^*}(x)$ for all $t>T^*$.
\end{enumerate}
\end{proposition}

\begin{proof}
The first point is a direct consequence of the even character of $A$ 
whereas the second point comes from the convexity of $A$.
Let us then prove the third point.
By convexity of the function $A$ and with \eqref{dynagg1}, we have
$$
m_1x'_1 \geq A \left(\frac{m_1}{2}+\sum_{j=2}^n \frac{m_j}{2} e^{x_1^0-x_j^0}\right)
-A \left(-\frac{m_1}{2}+\sum_{j=2}^n \frac{m_j}{2} e^{x_1^0-x_j^0}\right)>0,
$$
and
$$
m_nx'_n \leq A \left(-\sum_{j=1}^{n-1} \frac{m_j}{2} e^{x_j^0-x_n^0}+\frac{m_n}{2}\right)
-A \left(-\sum_{j=1}^{n-1} \frac{m_j}{2} e^{x_j^0-x_n^0}-\frac{m_n}{2}\right)<0.
$$
As for \eqref{dynagg2}, we can rewrite these last inequalities as~:
$$
x'_1(t) \geq a(\gamma_1(0))>0, \qquad x'_n\leq a(\gamma_n(0))<0.
$$
We deduce that there exists a time $T^*>0$ such that all masses collapse 
for $t=T^*$ in a single Dirac mass.
\end{proof}

\begin{remark}
Notice that we have in addition the following estimate for $T^*$:
$$
T^* < (x_n^0-x_1^0)/(a(\gamma_1(0))-a(\gamma_n(0))).
$$
\end{remark}

\begin{corollary}
Let us assume that $0\leq \rho^{ini}\in C_c(\RR)$ with compact support $[0,L]$.
Let us denote $\rho$ the duality solution of Theorem \ref{ExistFlux} 
with initial data $\rho^{ini}$.
Then there exists  $c^*\in [0,L]$ and $T^*>0$ such that 
$\rho(t,x)=\delta_{c^*}(x)$ for all $t>T^*$.
\end{corollary}

\begin{proof}
Let us approximate $\rho^{ini}$ by 
$$
\rho^{ini}_n(x)=\sum_{i=1}^n m_i^0 \delta_{x_i^0}(x),
$$
with $x_i^0=(i-1)L/n$, for $i=1,\dots,n$ and 
$m_i^0=\int_{x_i^0}^{x_{i+1}^0} \rho^{ini}(dx)$.
From Proposition \ref{propdynamic}, we deduce that there exists
$c_n^*\in [0,L]$ and $T_n^*>0$ such that the duality solution
of Theorem \ref{ExistFlux} with initial data $\rho_n^{ini}$ 
is such that $\rho_n(t,x)=\delta_{c^*_n}$ for all $t>T_n^*$.
Moreover, we have $T_n^* < L/(a(\gamma^n_1(0))-a(\gamma^n_n(0)))$
where we recall that
\beq\label{boundg1}
-m_1^0 + \sum_{j=1}^n \frac{m_j^0}{2} e^{-(j-1)L/n} < \gamma_1^n(0) <
\sum_{j=1}^n \frac{m_j^0}{2} e^{-(j-1)L/n},
\eeq
and 
\beq\label{boundgn}
-\sum_{j=1}^n \frac{m_j^0}{2} e^{(j-n)L/n} < \gamma_n^n(0) <
m_n^0-\sum_{j=1}^n \frac{m_j^0}{2} e^{(j-n)L/n}.
\eeq
By stability results on duality solutions in Theorem \ref{ExistDuality}
(see also subsection \ref{aggregat}), 
we deduce that $\rho_n\rightharpoonup \rho$ in $\smes$ as $n\to +\infty$.
Taking the limit in \eqref{boundg1} and \eqref{boundgn}, we deduce 
by continuity of $\rho^{ini}$ that
$$
\lim_{n\to +\infty} \gamma_1^n(0) = \int_0^L \rho^{ini}(x) e^{-x}\,dx
$$
and
$$
\lim_{n\to +\infty} \gamma_n^n(0) = -\int_0^L \rho^{ini}(x) e^{-L+x}\,dx.
$$
Moreover, since $\rho^{ini}$ is continuous with compact support in $[0,L]$
we have $\rho^{ini}(0)=\rho^{ini}(L)=0$.
We deduce that the sequence $(T^*_n)_{n\in \NN^*}$ is bounded.
Thus there exists a time $T^*$ independent of $n$ such that
$\rho_n(t)=\delta_{c_n^*}$ for all $t>T^*$. Taking the limit when 
$n\to +\infty$, we conclude that 
there exists $c\in [0,L]$ such that $\rho(t)=\delta_c$ for all $t>T^*$.
\end{proof}

\begin{remark}\label{rq:a=id}
Taking $a=Id$, therefore $A(x)=x^2/2$, we deduce from \eqref{dynagg1} that 
$$
x'_i = \sum_{j\neq i} m_j \pa_xK(x_j-x_i).
$$
We recover the dynamics of the aggregation equation as noticed by 
Carrillo et al. in \cite{Carrillo}. These authors prove in particular
the concentration in finite time of the total mass in the center
of mass. In the framework of the present work, which is focused on applications 
to chemotaxis, $a$ is not assumed to be the identity function, so that
 the center of mass is not conserved. A numerical evidence of this 
phenomenon will be proposed in the last subsection of this paper.
\end{remark}

\subsection{Discretization}
%%%%%%%%%%%%%%%%%%%%%%%%%%%%%%%%%%%%%%%

The numerical resolution of system \eqref{eqrhohydro}--\eqref{eqShydro}
is far from obvious. A first naive idea consists in applying a standard splitting method where we treat
separately the scalar conservation law \eqref{eqrhohydro} and 
the elliptic equation \eqref{eqShydro}.  It turns out that such a scheme is unable to 
recover the correct definition of the flux and therefore of the product $a(\pa_xS)$ by $\rho$. In particular,
it leads to stationary Dirac masses.

A second idea consists in solving the distributional conservation law
\eqref{eqrhodis} by a finite volume method. 
It involves a discretization of the flux $J$ on the interface 
of each cell of the mesh, and thus one could expect a correct computation of the flux, and
therefore a convenient interpretation of the product.
However, this definition of the flux involves the calculation of two
derivatives of $S$. Using a centered scheme to discretize this quantity
induces spurious oscillations as it is usually noticed for centered 
scheme on scalar conservation laws. We can then upwind the scheme depending
on the sign of $a(\pa_xS)$ computed at previous iteration. 
But in doing so, we actually specify a value for $a(\pa_xS)$ in the definition 
of the product $a(\pa_xS)$ with $\rho$, and this can lead to capture wrong solutions.

Next, one can think of solving the equation \eqref{eqSdistrib} on $S$, motivated by the fact that it
plays a key part in the uniqueness, and that $\rho$ can be recovered readily from $S$. 
However the equation is non local and its numerical resolution appears to be 
quite complicated and with a high computational cost (even in the one dimensional setting). 

Thus we prefer to use a method based on 
the dynamics of aggregates, detailed in Section \ref{aggregat}.
We use the principle of a particle method
in which we approximate the density by a sum of Dirac masses.
Then the motion of these pseudo-particles is approximated by
discretizing system \eqref{dynagg} with an explicit Euler scheme.
More precisely, let us assume that we have an approximation 
of $\rho$ at time $t_n=n\Delta t$, given by
	\beq\label{rhodiscret}
\rho^n(x)=\sum_{i=1}^{I^n} m_i^n \,\delta_{y_i^n}(x),
	\eeq
where $m_i^n>0$ is the mass allocated to the pseudo-particle at the position 
$y_i^n$ with $y_1^n<y_2^n<\dots<y_{I^n}^n$ for $I^n\in \NN^*$.
Then an approximation of the potential at time $t^n$ is given by
	$$
S^n(x) = \sum_{i=1}^{I^n} m_i^n \,e^{-|x-y_i^n|}.
	$$
Using an explicit Euler scheme, we compute the new position
	\begin{align*}
y_i^{n+1}= & y_i^n + \frac{\Delta t}{m_i^n}
A\left(-\sum_{j=1}^{i-1} \frac{m_j^n}{2} e^{y_j^n-y^n_i}
+\frac{m_i^n}{2}+\sum_{j=i+1}^{I^n} \frac{m_j^n}{2} e^{y^n_i-y^n_j}\right) \\
 & - \frac{\Delta t}{m_i^n} A\left(-\sum_{j=1}^{i-1} \frac{m_j^n}{2} e^{y^n_j-y^n_i}-
\frac{m_i^n}{2}+\sum_{j=i+1}^{I^n} \frac{m_j^n}{2} e^{y^n_i-y^n_j}\right).
	\end{align*}
Next, we test if some pseudo-particles have collided during the time step
$\Delta t$. If $y_{j+1}^{n+1}\leq y_j^{n+1}$ for $j\geq 1$, then the
pseudo-particles $j$ and $j+1$ have collapsed and form a unique
pseudo-particle which has the mass $m_j^n+m_{j+1}^n$. 
In this case, we decide to set this pseudo-particle at the position
$\frac 12 (y_{j+1}^{n+1}+ y_j^{n+1})$ and set $m_j^{n+1}=m_j^n+m_{j+1}^n$,
moreover we have therefore $I^{n+1}=I^n-1$.
Finally, for given initial sequences $(y^0_i)_{i=1,\dots,I^0}$ and 
$(m^0_i)_{i=1\dots,I^0}$ of size $I^0$, 
we can construct $(y^n_i)$ and $(m^n_i)$ of size $I^n$ as above.

Using well-known result on the convergence of Euler scheme, we 
deduce that, for given initial data $(y^0_i)_{i=1,\dots,I^0}$, 
$(m^0_i)_{i=1\dots,I^0}$ and $I^0$, $y_i^n$ defined above 
converges to the solution $x_i(t)$ of \eqref{dynagg} 
when $\Delta t$ tends to $0$ such that $t_n\to t$.
Using the convergence result in Section \ref{sub:existDual}, we deduce that
the function $\rho^n$ in \eqref{rhodiscret} converges in $\smes$
to the unique duality solution of Theorem \ref{ExistFlux}.
Then the method introduced above is convergent provided we 
discretize the initial data $\rho^{ini}$ in such a way that
$\rho^0(x):=\sum_{i=1}^{I^0} m_i^0 \,\delta_{y_i^0}(x)$ 
converges in $\calM_b$ to $\rho^{ini}$. 
Moreover, we verify easily that we have 
	$$
\sum_{i=1}^{I^0} m_i^0 = \sum_{i=1}^{I^n} m_i^n,\quad\mbox{ and }\ 
I^n\leq I^0,\ \mbox{ for all } n\in \NN,
	$$
and that the approximation $\rho^n$ of $\rho(t_n)$ is nonnegative.

\subsection{Numerical results}
%%%%%%%%%%%%%%%%%%%%%%%%%%%%%%%%%%%%%
In this Section, we present numerical simulations of model 
\eqref{eqrhohydro}--\eqref{eqShydro} using the algorithm
described above. 
We first approximate the initial data $\rho^{ini}\geq 0$, 
which is assumed to be compactly supported for numerical purpose,
in the following way: we introduce a discretization 
$x_j=x_0+j\Delta x$ of the bounded domain which includes the compact support
of $\rho^{ini}$ and we define
$$
m_i^0=\int_{x_i-\frac{\Delta x}{2}}^{x_i+\frac{\Delta x}{2}}
\rho^{ini}(x)\,dx.
$$
Then the sequence $(y_j^0)_j$ is defined by the nodes $(x_i)$ for 
which $m_i^0$ is not zero, and $I^0$ correspond to the number of 
$i\in \NN$ such that $m_i^0$ is not zero. We construct then 
the approximation of $\rho^{ini}$ by
$$
\rho^0(x):=\sum_{i=1}^{I^0} m_i^0 \,\delta_{y_i^0}(x).
$$ 

\begin{figure}[!ht]
 \includegraphics[width=5.5cm,height=6cm]{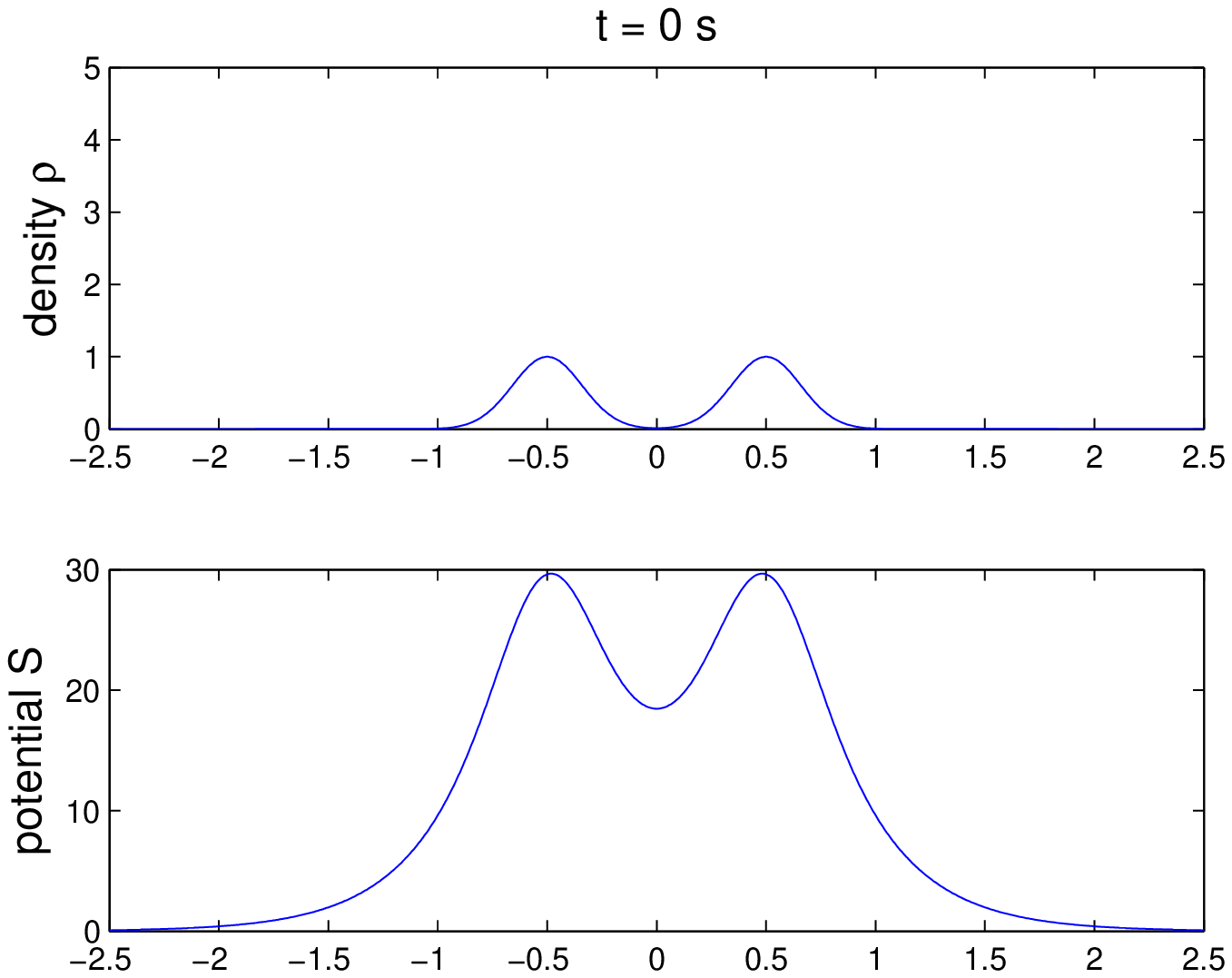}
 \includegraphics[width=5.5cm,height=6cm]{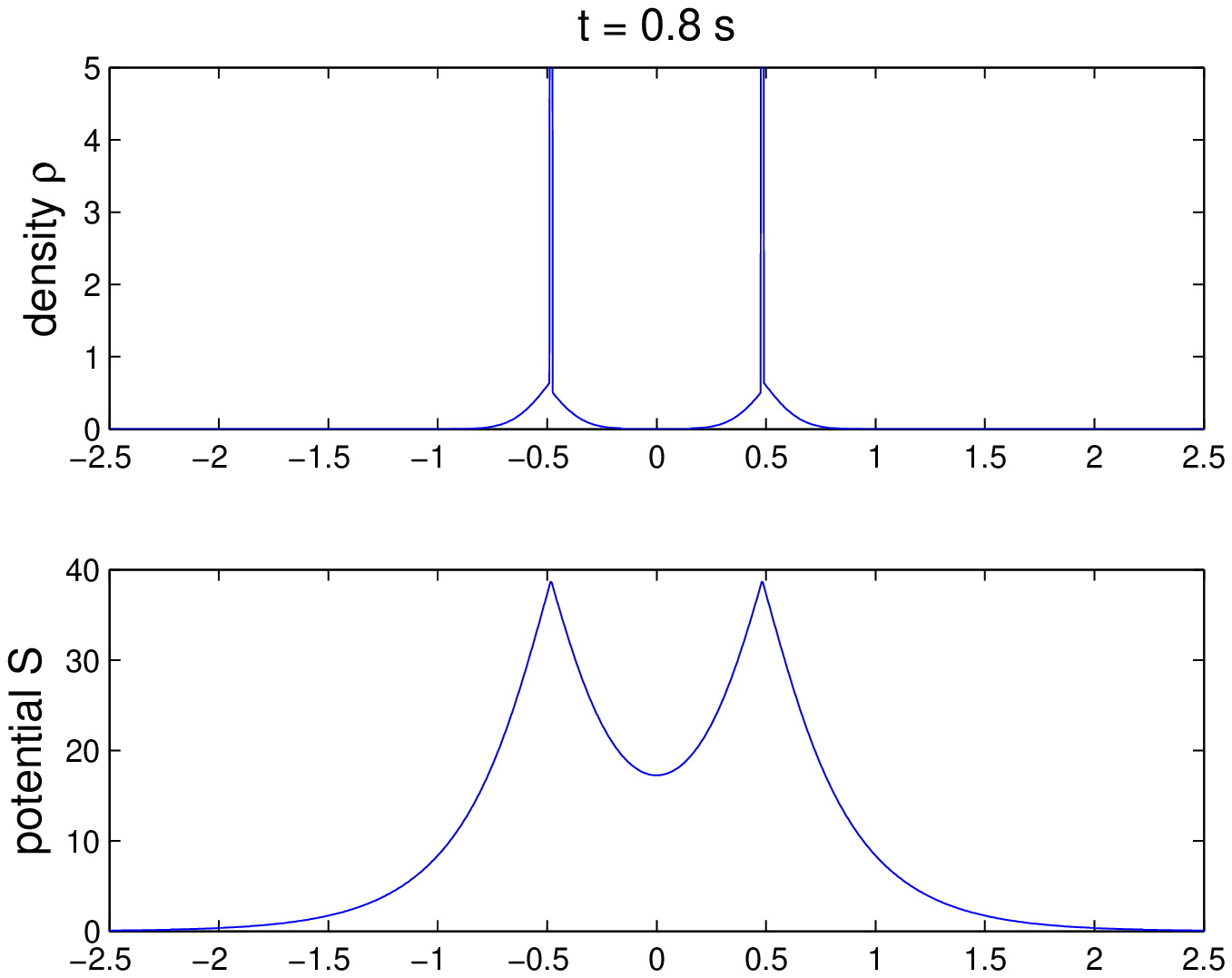}
 \includegraphics[width=5.5cm,height=6cm]{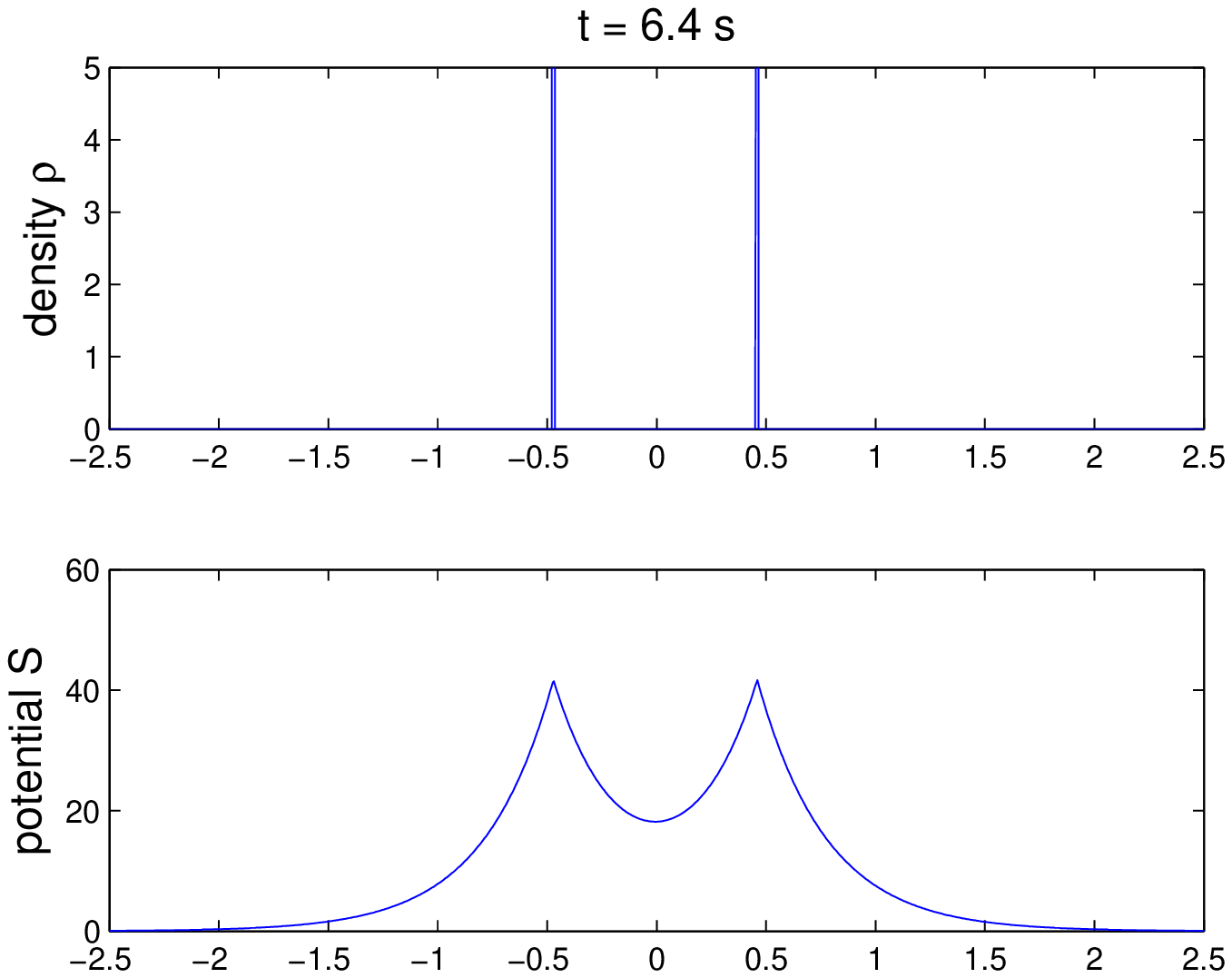}\\[3mm]
 \includegraphics[width=5.5cm,height=6cm]{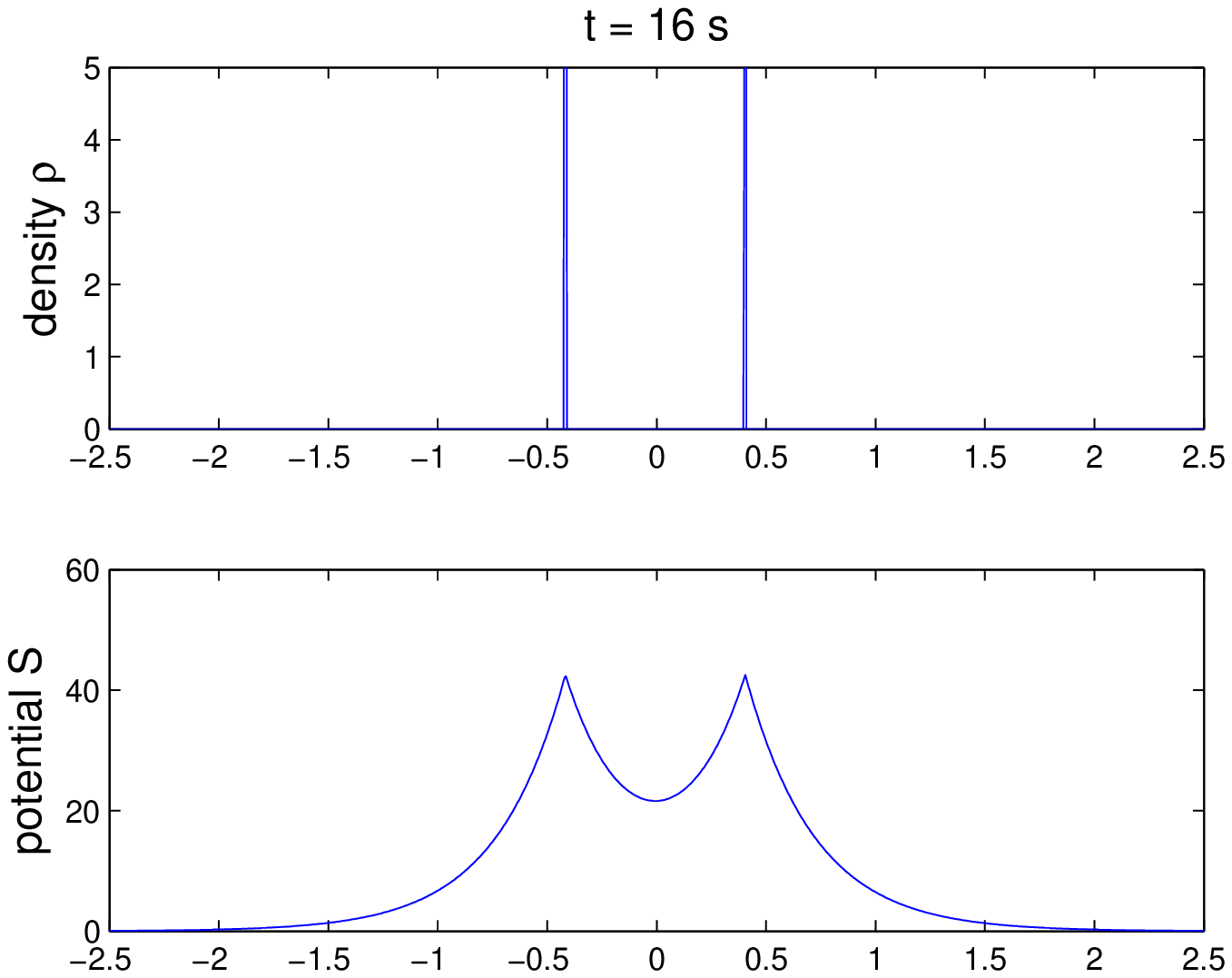}
 \includegraphics[width=5.5cm,height=6cm]{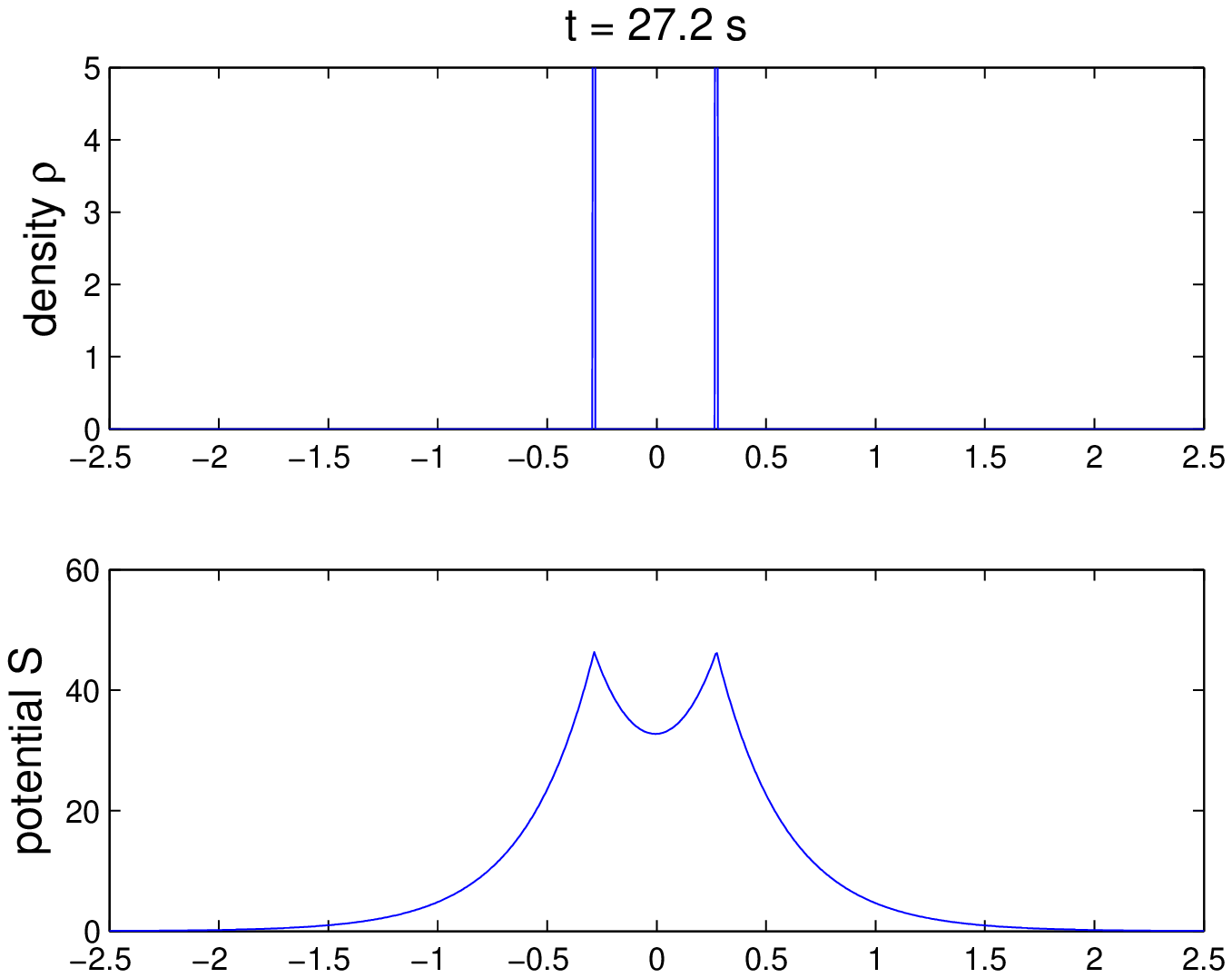}
 \includegraphics[width=5.5cm,height=6cm]{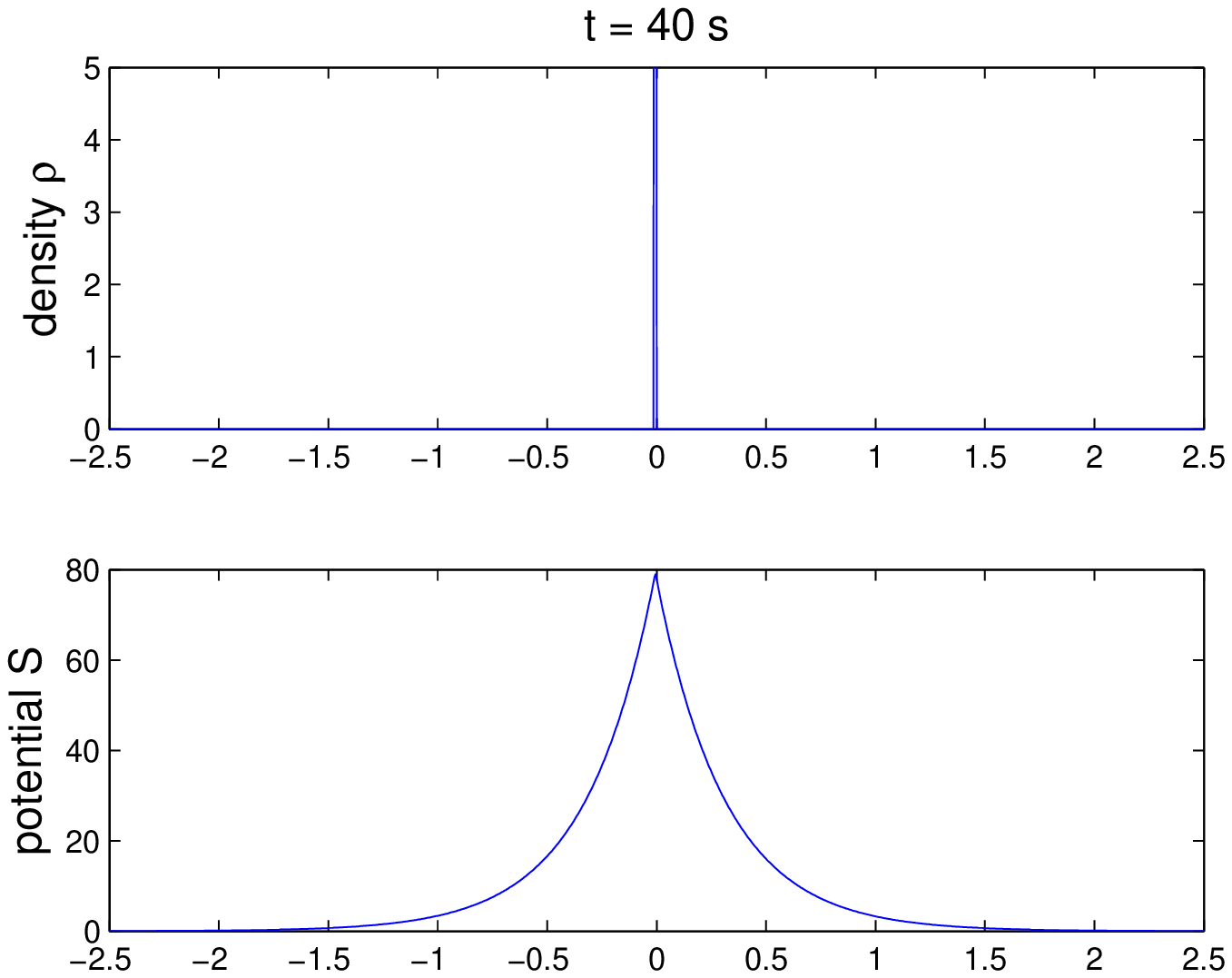}
\caption{Dynamics of the density $\rho$ (top) and of the potential $S$ (bottom)
for an initial density given by the sum of two Gaussian.}
\label{2bump}
\end{figure}
We present in Figure \ref{2bump} the dynamics of the density $\rho$ 
and of the chemoattractant concentration $S$
for an initial data $\rho^{ini}$ given by the sum of two 
Gaussian functions, more precisely
$$
\rho^{ini}(x)=e^{-20(x-0.5)^2}+e^{-20(x+0.5)^2}.
$$
As expected, we first observe the formation of two Dirac masses
at the position where $\pa_xS$ initially vanishes. Then, 
the two aggregates collapse in the center.
Looking at the time evolution, we notice that the first step of formation of
aggregates is fast compared to the time of collapse.

\begin{figure}[!ht]
\begin{center}
 \includegraphics[width=5.5cm,height=6cm]{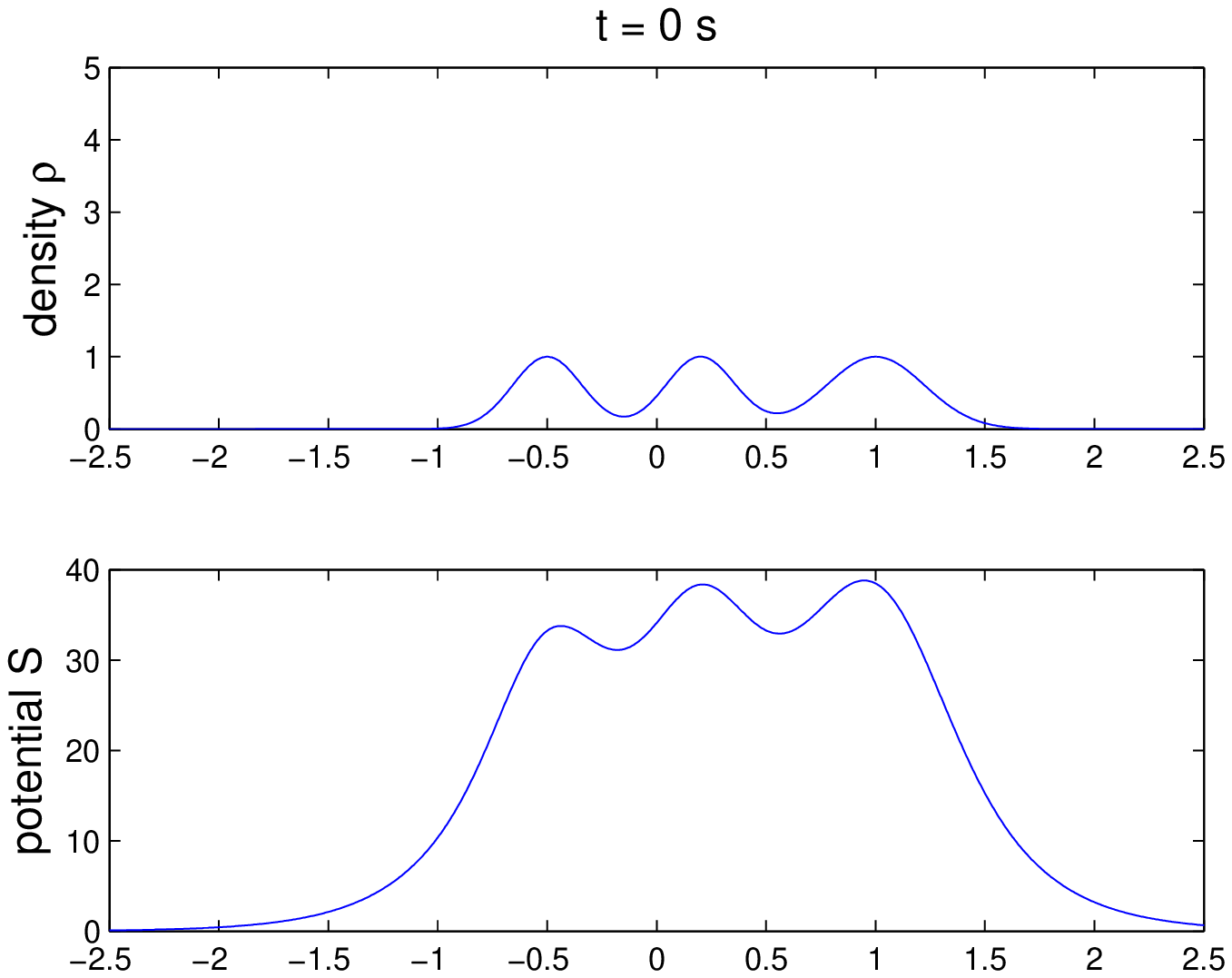}
 \includegraphics[width=5.5cm,height=6cm]{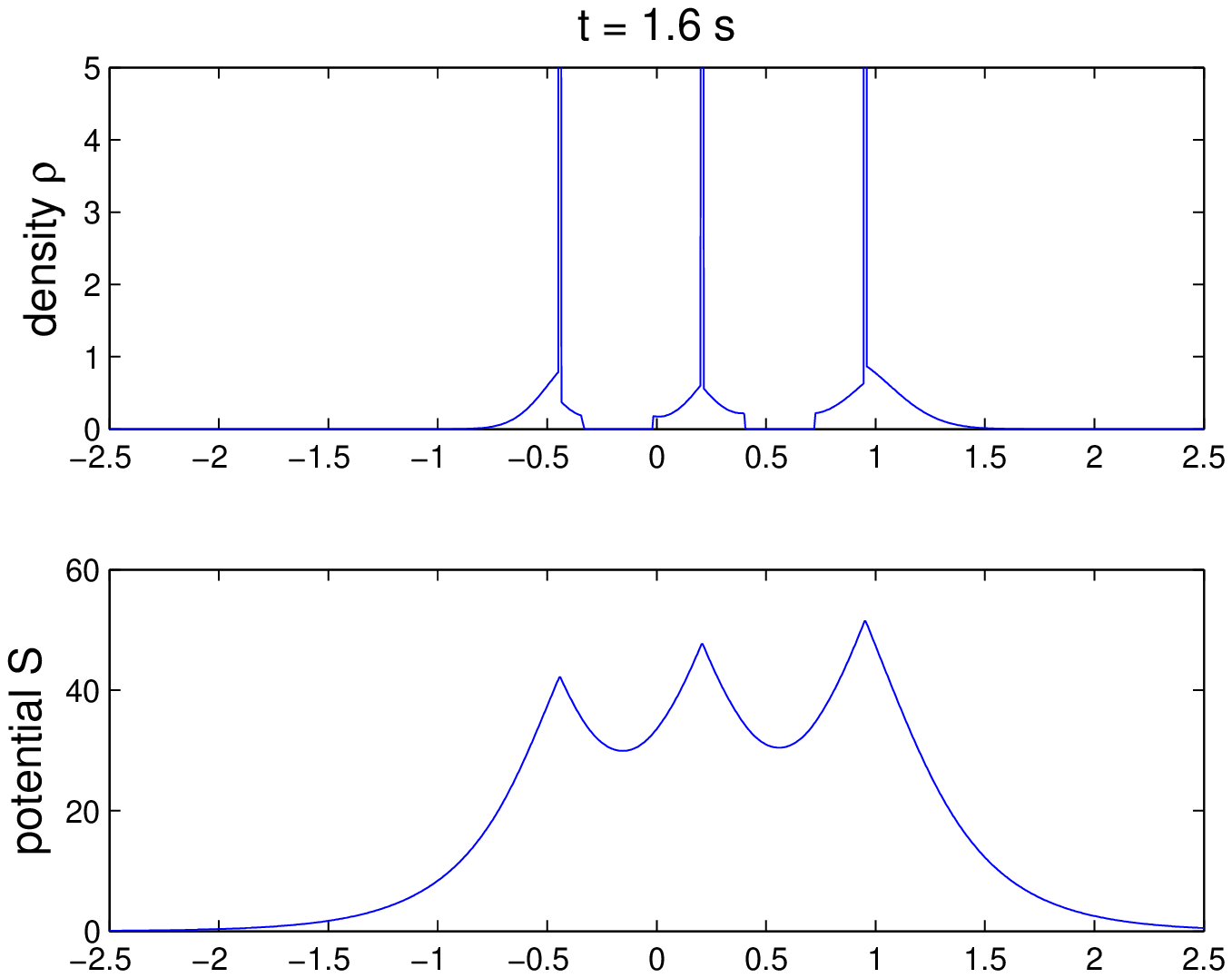}
 \includegraphics[width=5.5cm,height=6cm]{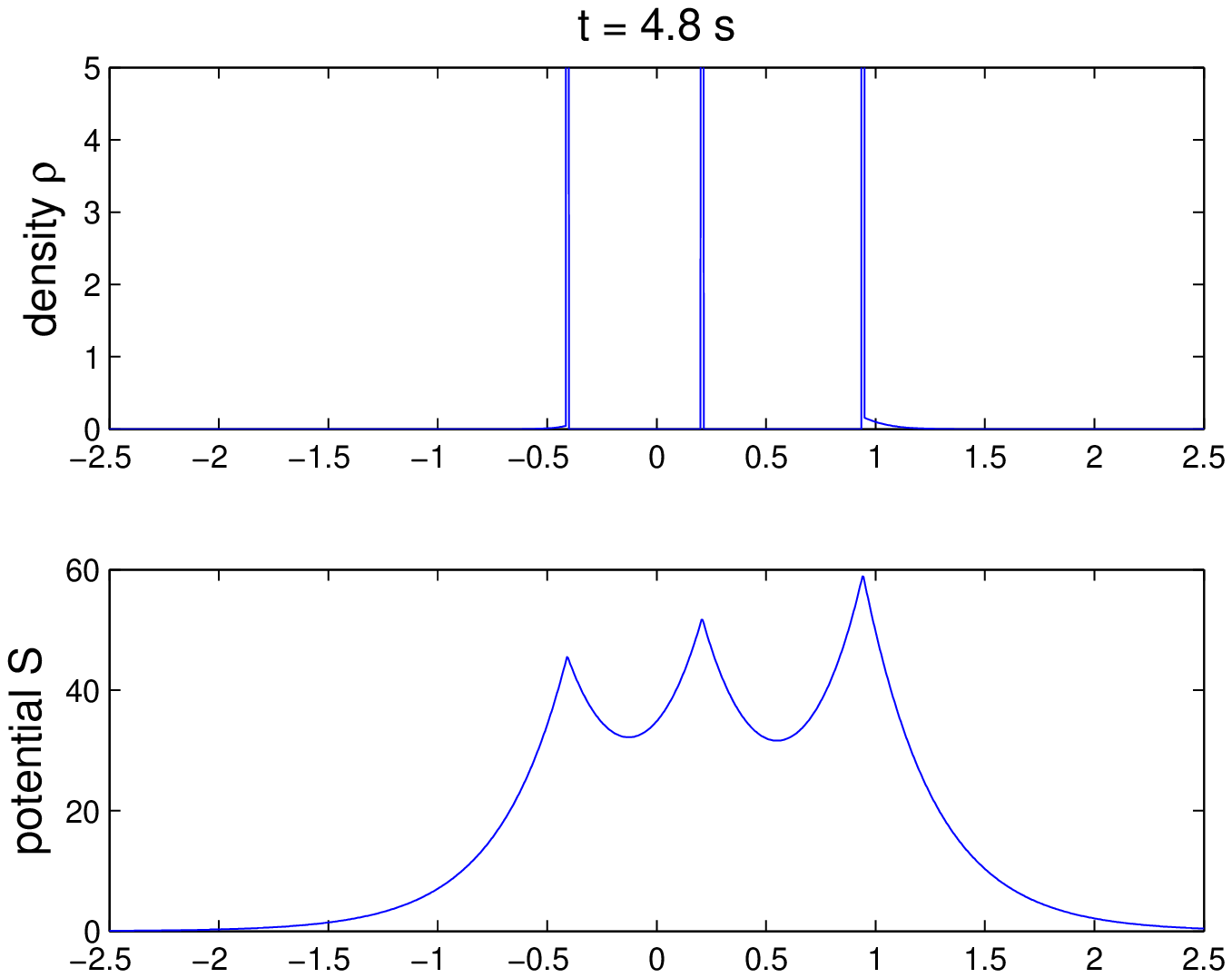}\\[3mm]
 \includegraphics[width=5.5cm,height=6cm]{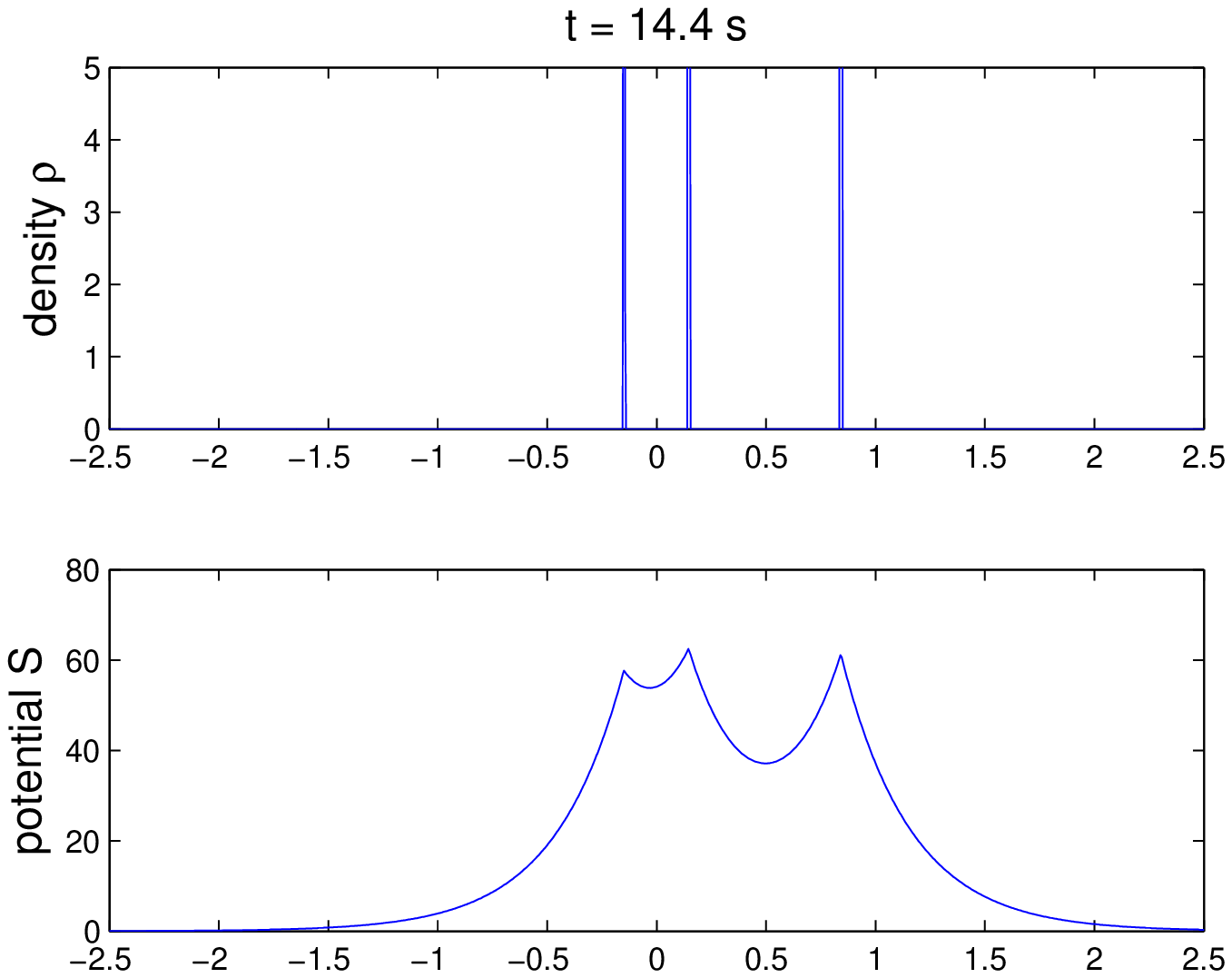}
 \includegraphics[width=5.5cm,height=6cm]{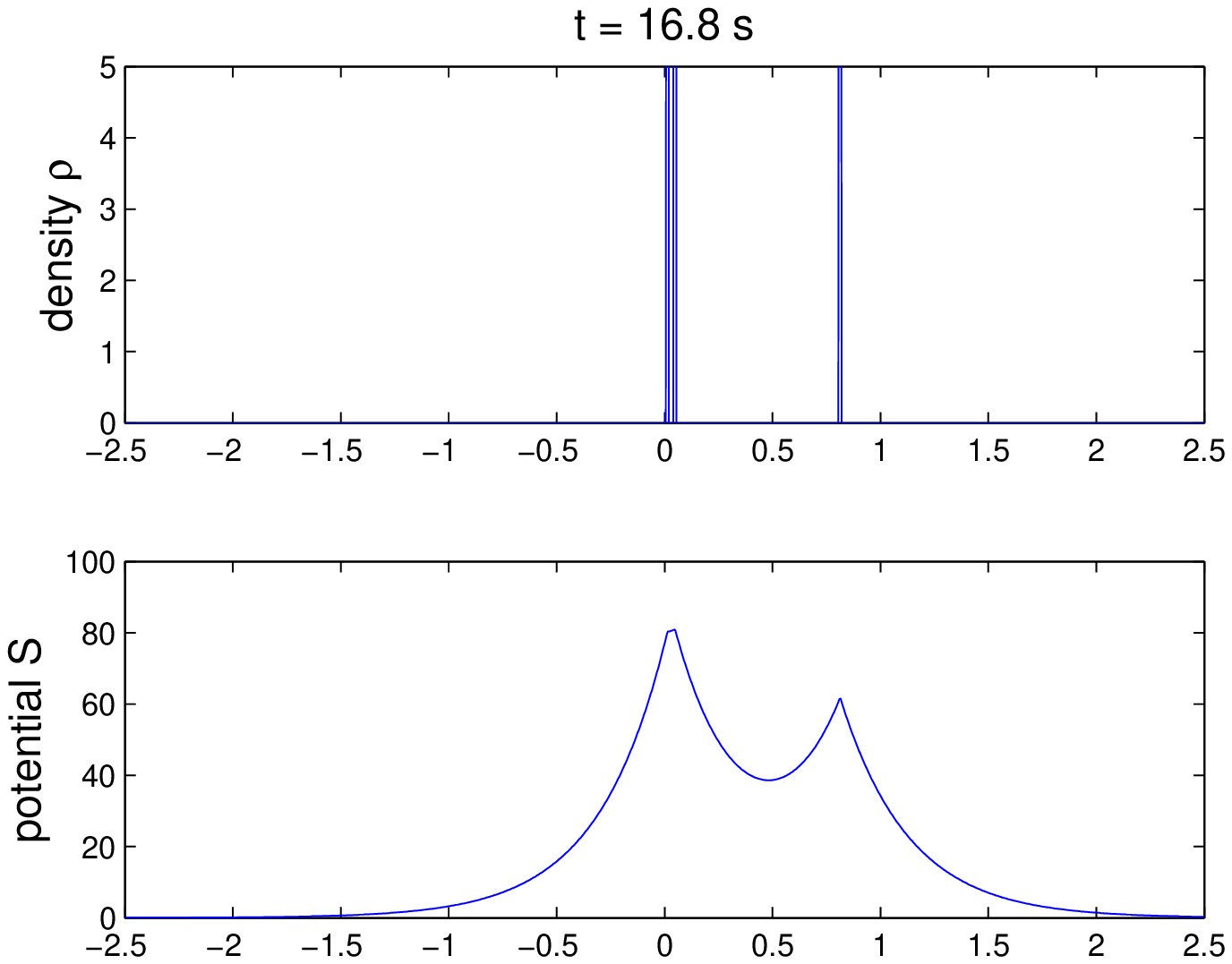}
 \includegraphics[width=5.5cm,height=6cm]{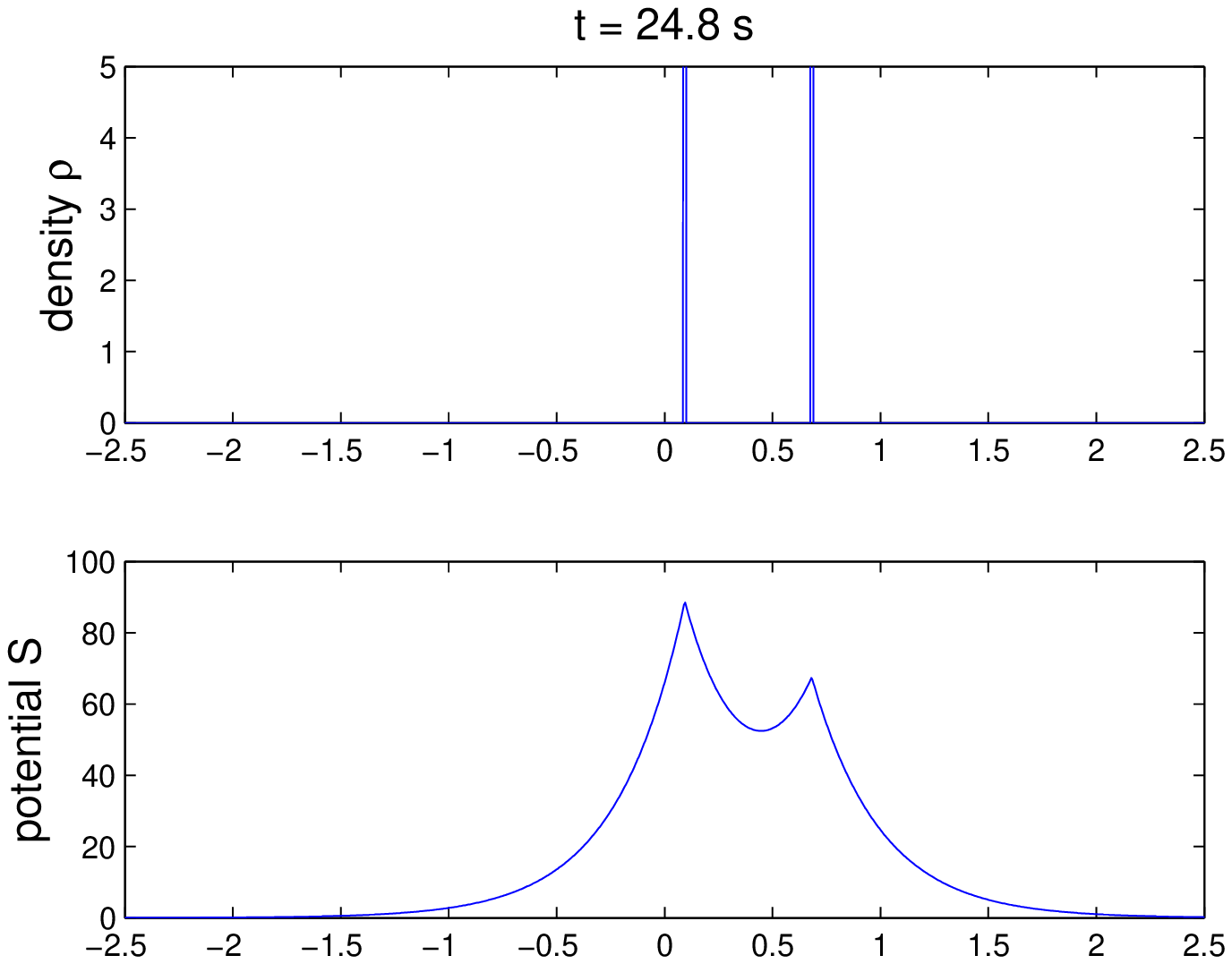}\\[3mm]
 \includegraphics[width=5.5cm,height=6cm]{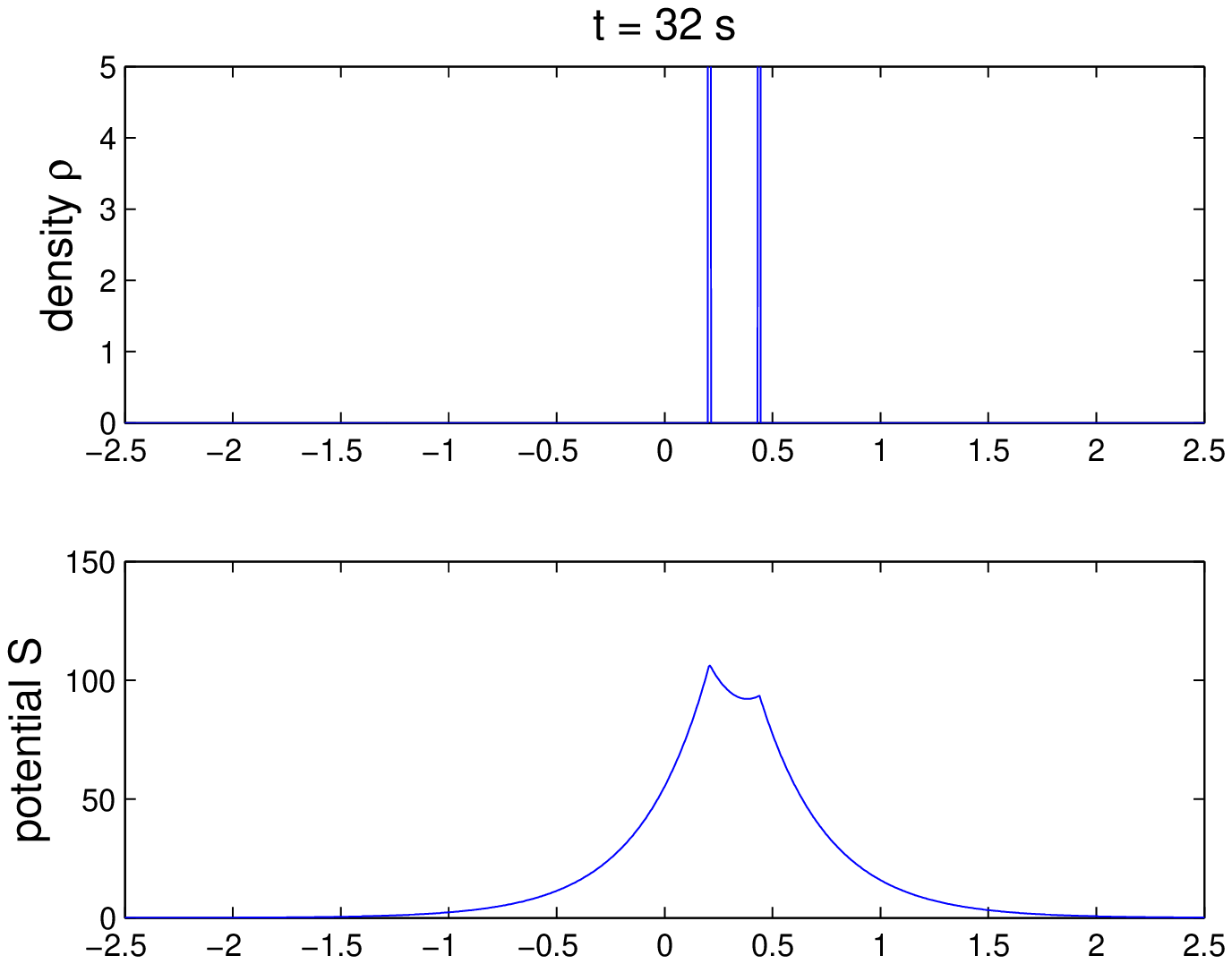}
 \includegraphics[width=5.5cm,height=6cm]{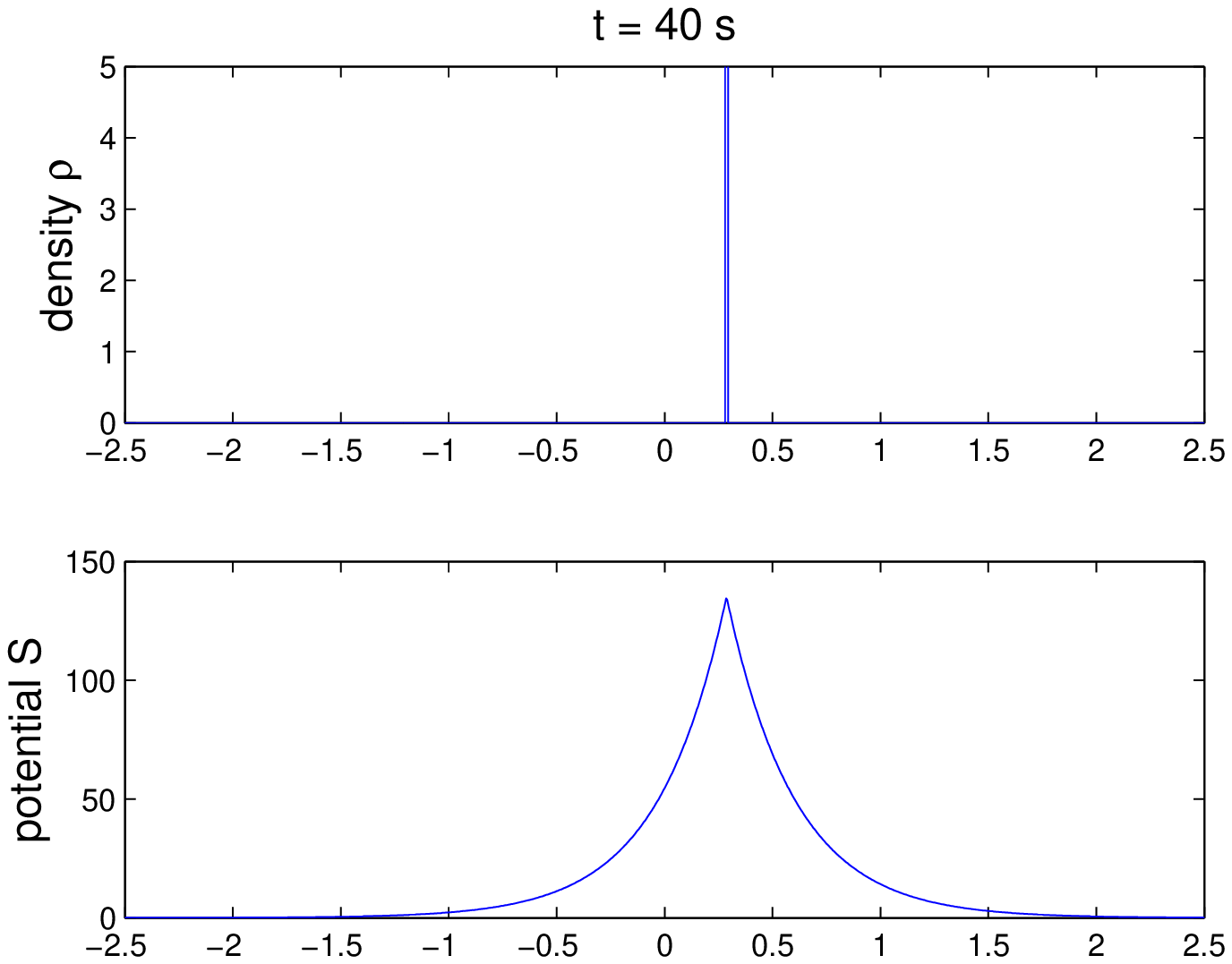}
\end{center}
\caption{Dynamics of the density $\rho$ (top) and of the potential $S$ (bottom)
for an initial density given by the sum of three Gaussian.}
\label{3bump}
\end{figure}
In Figure \ref{3bump} we display the dynamics for an initial data given 
by the sum of three Gaussian functions:
	$$
\rho^{ini}(x)=e^{-10(x-1)^2}+e^{-20(x-0.2)^2}+e^{-20(x+0.5)^2}.
	$$
We observe the formation of three Dirac masses that moves according to 
the dynamical system \eqref{dynagg1}. They collapse then in finite
time.

\begin{figure}[!ht]
 \includegraphics[width=5.5cm,height=6cm]{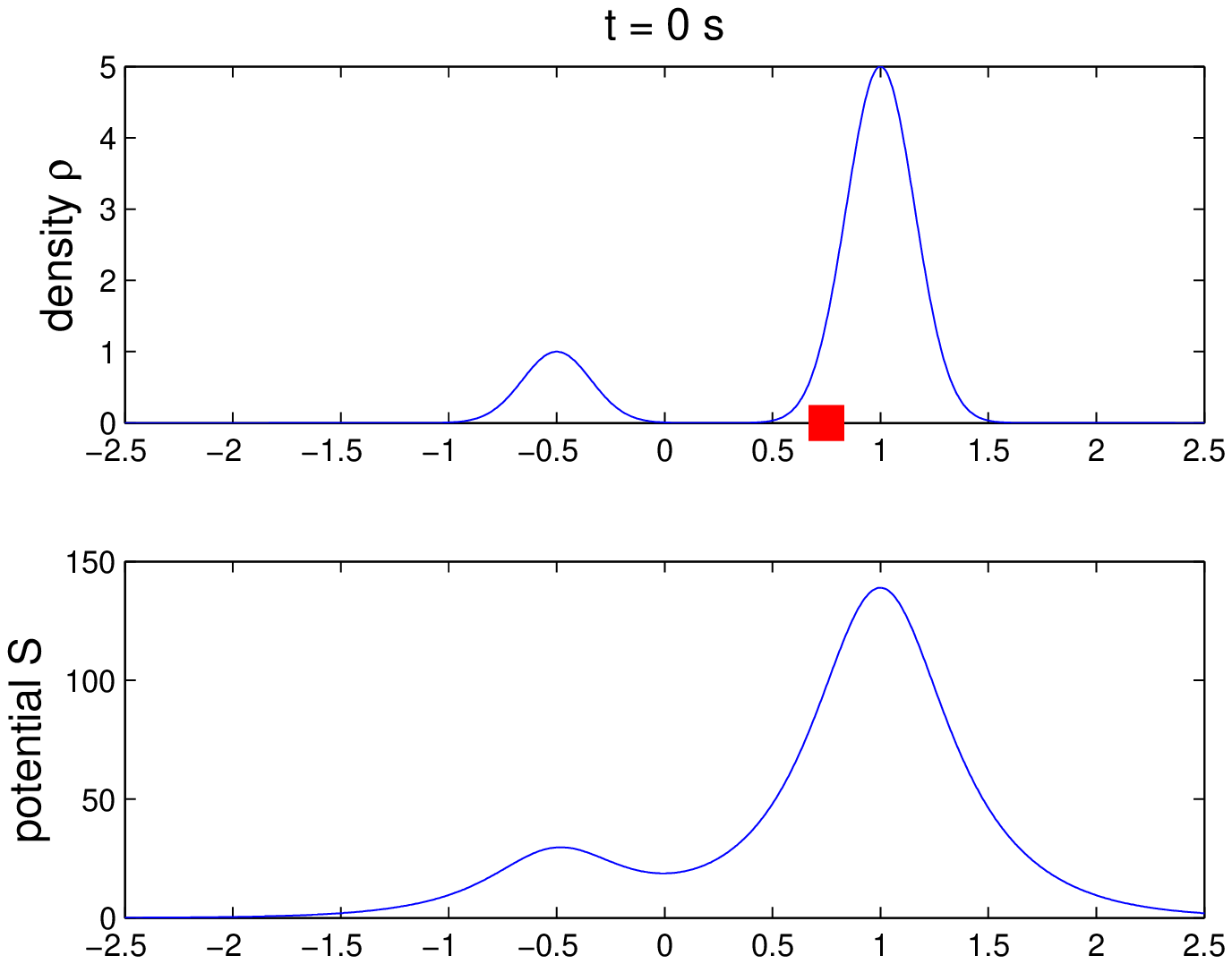}
 \includegraphics[width=5.5cm,height=6cm]{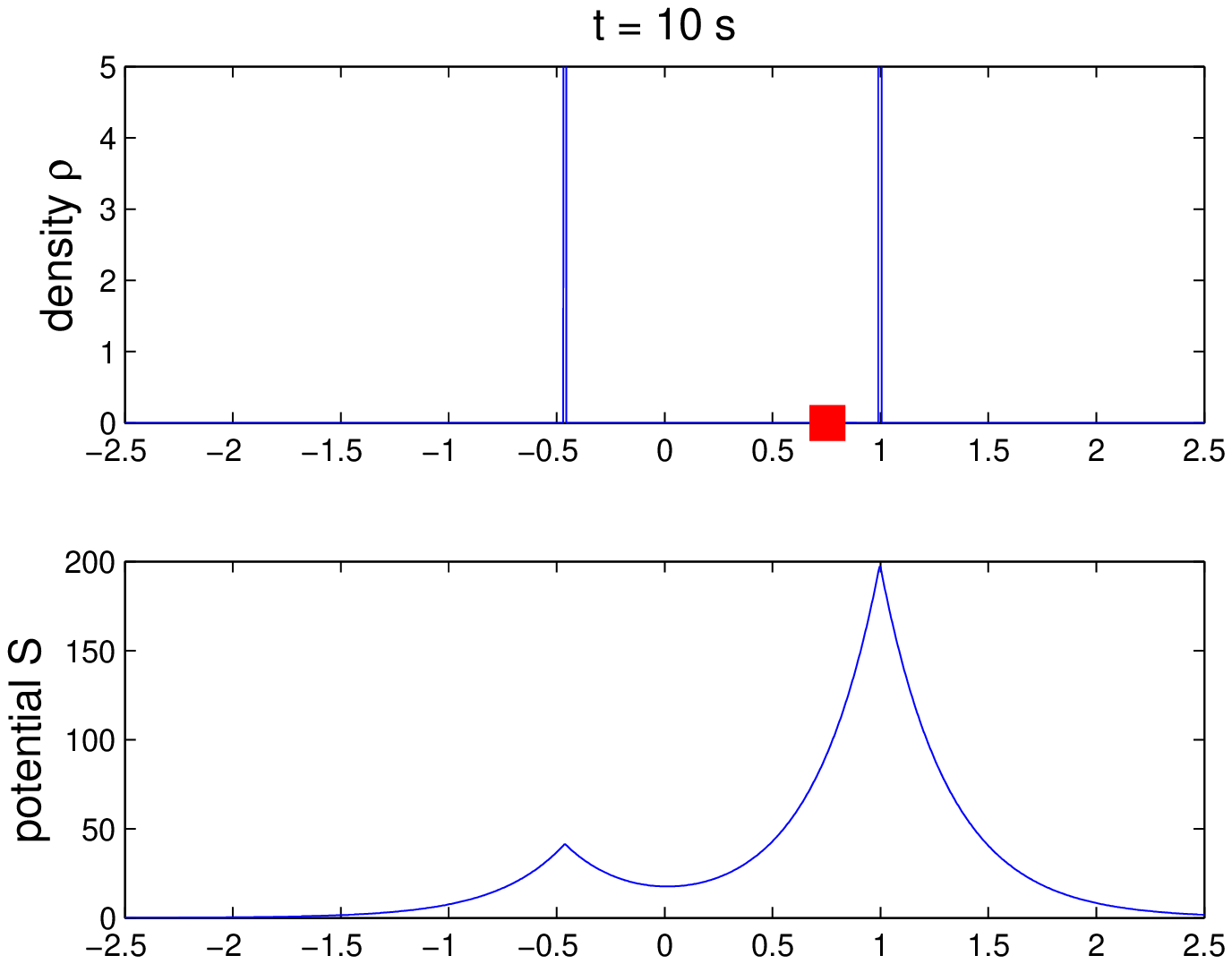}
 \includegraphics[width=5.5cm,height=6cm]{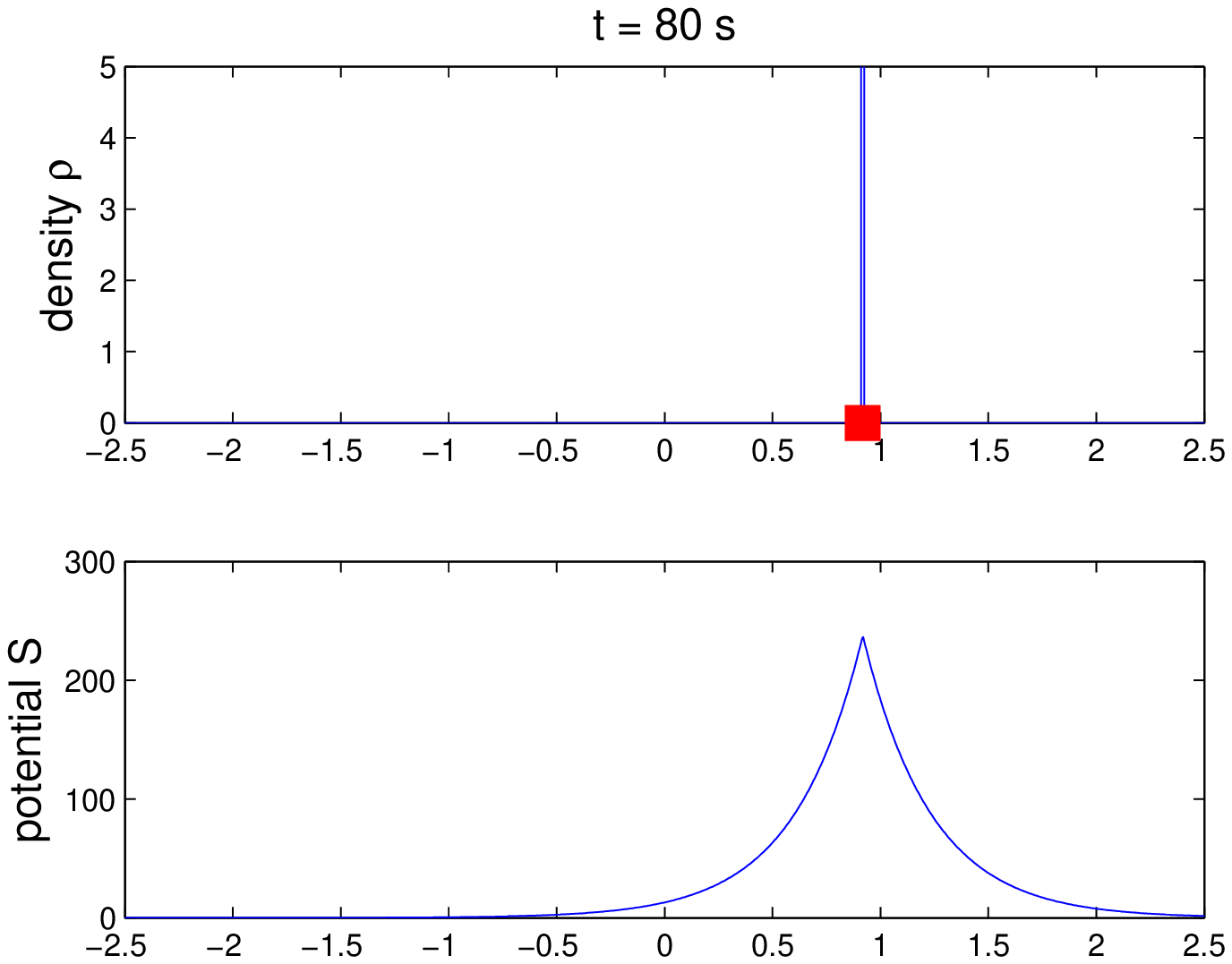}
\caption{Dynamics of the density $\rho$ (top) and of the potential $S$ (bottom) with
the dynamics of the center of mass represented by a red square.
The center of mass moves.}
\label{cm}
\end{figure}
Finally, as we have already noticed, we evidence that 
the center of mass is not fixed.
For instance, Figure \ref{cm} represents the dynamics of the density
and of the potential for an initial data made of one big bump with 
one small bump:
$$
\rho^{ini}(x)=5 e^{-20(x-1)^2}+0.5 e^{-20(x+0.5)^2}.
$$
The square shows the time dynamics of the center of mass.
We observe that the center of mass at the final time 
is not located at the same position as at the initial time.

%%%%%%%%%%%%%%%%%%%%%%%%%%%%%%%%%%%%%%%%%%%%%%%%%%%%%%%%%%%%%%%%%%%%%%%%%%%%%%%%


\begin{thebibliography}{99}
\bs

\bibitem{alt} W. Alt, {\it Biased random walk models for chemotaxis and 
related diffusion approximations}, J. Math. Biol. {\bf 9}, 147--177 (1980).

\bibitem{ambrosioBV} L. Ambrosio, N. Fusco, D. Pallara, {\em Functions of
    Bounded Variation and Free Discontinuity Problems},
  Oxford University Press, 2000. 

\bibitem{bertozzi2} A.L. Bertozzi, J.A. Carrillo, Th. Laurent, {\it Blow-up
  in multidimensional aggregation equation with mildly singular interaction
kernels}, Nonlinearity {\bf 22} (2009) 683--710.

\bibitem{bj1} F. Bouchut, F. James, {\it One-dimensional transport equations 
    with discontinuous coefficients},
  Nonlinear Analysis TMA {\bf 32}  (1998), n$^o$ 7, 891--933.

\bibitem{BJpg} F. Bouchut, F. James,
{\it Duality solutions for pressureless gases, monotone scalar
conservation laws, and uniqueness},
Comm. Partial Differential Eq., {\bf 24} (1999), 2173-2189.

\bibitem{bjm} F. Bouchut, F. James, S. Mancini,
  {\it Uniqueness and weak stability for multidimensional transport 
    equations with one-sided Lipschitz coefficients},
  Ann. Scuola Norm. Sup. Pisa Cl. Sci. (5), {\bf IV} (2005), 1-25.

\bibitem{bourncalv} N. Bournaveas, V. Calvez, S. Guti\`errez, B. Perthame,
  {\it Global existence for a kinetic model of chemotaxis via dispersion and
  Strichartz estimates}, Comm. Partial Differential Eq., {\bf 33} (2008), 79--95.

\bibitem{Carrillo} J. A. Carrillo, M. DiFrancesco, A. Figalli, T. Laurent, D. Slep\v{c}ev,
{\it Global-in-time weak measure solutions and finite-time aggregation for
nonlocal interaction equations}, Duke Math. J. {\bf 156} (2011), 229--271.

\bibitem{chalubperth} F.A.C.C. Chalub, P.A. Markowich, B. Perthame, 
  C. Schmeiser,  {\it Kinetic models for chemotaxis and their 
    drift-diffusion limits}, Monatsh. Math.  {\bf 142} (2004), 123--141.

\bibitem{dolschmeis} Y. Dolak, C. Schmeiser, {\it Kinetic models for
chemotaxis: Hydrodynamic limits and spatio-temporal mechanisms}, J. Math.
Biol. {\bf 51}, 595--615 (2005).

\bibitem{erbanhwang} R. Erban, H.J. Hwang, {\it Global existence results for 
complex hyperbolic models of bacterial chemotaxis}, Disc. Cont. Dyn. 
Systems - Series B, {\bf 6} (2006), n$^o$ 6, 1239--1260.

\bibitem{erbanothmerbacterie} R. Erban, H.G. Othmer, {\it From individual 
    to collective behavior in bacterial chemotaxis}, SIAM J. Appl. Math. 
  {\bf 65} (2004/05), n$^o$ 2, 361--391.

\bibitem{filblaurpert} F. Filbet, Ph. Lauren\c{c}ot, B. Perthame, {\it 
Derivation of hyperbolic models for chemosensitive movement}, J. Math. Biol.
{\bf 50} (2005), 189--207.

\bibitem{hillenothmer} T. Hillen, H.G. Othmer, {\it The diffusion 
    limit of transport equations derived from velocity jump processes},
  SIAM J. Appl. Math. {\bf 61}  (2000), n$^o$ 3, 751--775.

\bibitem{hwang} H.J. Hwang, K. Kang, A. Stevens, {\it Global solutions of 
  nonlinear transport equations for chemosensitive movement}, SIAM J. Math.
Anal. {\bf 36} (2005), n$^o$ 4, 1177--1199.

\bibitem{note} F. James, N. Vauchelet, {\it A remark on duality solutions 
for some weakly nonlinear scalar conservation laws}, 
C. R. Acad. Sci. Paris, S\'er. I 349 (2011), 657-661, doi:10.1016/j.crma.2011.05.004

\bibitem{jamesnv} F. James, N. Vauchelet, {\it On the hydrodynamical limit 
    for a one dimensional kinetic model of cell aggregation by chemotaxis},
  to appear in Riv. Mat. Univ. Parma.

\bibitem{NPS} J. Nieto, F. Poupaud, J. Soler, {\it High field limit for
  Vlasov-Poisson-Fokker-Planck equations}, Arch. Rational Mech. Anal. {\bf 158}
(2001), 29--59.

\bibitem{NPSm3as} J. Nieto, F. Poupaud, J. Soler, {\it About uniqueness 
    of weak solutions to first order quasi-linear equations}, 
  Math. Models Methods Appl. Sci. {\bf 12} (2002), no. 11, 1599–1615.

\bibitem{othdunalt} H.G. Othmer, S.R. Dunbar, W. Alt, {\it Models of 
    dispersal in biological systems}, J. Math. Biol. {\bf 26} (1988), 263--298.

\bibitem{othhill} H.G. Othmer, T. Hillen, {\it The diffusion limit of transport 
    equations. II. Chemotaxis equations}, SIAM J. Appl. Math. {\bf 62} (2002),  1222--1250.

\bibitem{othste} H.G. Othmer, A. Stevens, {\it Aggregation, blowup, 
    and collapse: the ABCs of taxis in reinforced random walks}, 
  SIAM J. Appl. Math. {\bf 57} (1997), 1044--1081.

\bibitem{perthame} B. Perthame, {\it PDE models for chemotactic movements: 
parabolic, hyperbolic and kinetic}, Appl. Math. {\bf 49} (2004), n$^o$ 6, 539--564.

\bibitem{perthame2} B. Perthame, {\it Transport Equations in Biology}, 
  Frontiers in Mathematics. Basel: Birk\"auser Verlag.

\bibitem{poupaud} F. Poupaud, {\it Diagonal defect measures, adhesion dynamics
  and Euler equation}, Meth. Appl. Anal. {\bf 9} (2002), 533--561. 

\bibitem{pouras} F. Poupaud, M. Rascle,
  {\it Measure solutions to the linear multidimensional transport equation
    with discontinuous coefficients},
  Comm. Partial Diff. Equ. {\bf 22} (1997), 337--358.

\bibitem{nv} N. Vauchelet, {\it Numerical simulation of a 
  kinetic model for chemotaxis}, Kinetic and Related Models {\bf 3} (2010),
n$^o$ 3, 501--528.

\bibitem{volpert} A.I. Vol'pert, {\it The spaces BV and quasilinear equations},
  Math. USSR Sb., {\bf 2} (1967), 225--267.

\end{thebibliography}
\end{document}